\UseRawInputEncoding
\documentclass[11pt, a4paper]{article}
\usepackage[dvipsnames]{xcolor}
\usepackage{amsmath,amssymb,amsthm,url,mathtools}
\usepackage{extpfeil} 

\usepackage[lmargin=35mm,rmargin=35mm,bmargin=30mm,tmargin=30mm]{geometry}
\usepackage{bbm}
\usepackage[inline]{enumitem}
\usepackage{subcaption}

\usepackage{esvect} 
\usepackage{hyperref}
\usepackage{parskip}
\usepackage{tikz}
\usetikzlibrary{arrows}
\usetikzlibrary{calc}
\usetikzlibrary{patterns}
\usetikzlibrary{shapes}
\usetikzlibrary{positioning}
\usetikzlibrary{decorations.pathmorphing}
\usepackage{tikz-cd} 

\usepackage{standalone}
\usepackage[capitalise]{cleveref}

\usepackage[longnamesfirst,numbers,sort]{natbib}

\hypersetup{
  colorlinks   = true,
  urlcolor     = blue!50!black,
  linkcolor    = blue!50!black,
  citecolor   = blue!50!black
}

\title{Boundary rigidity of $3$D \cato\ cube complexes}
\author{
John Haslegrave\footnote{Mathematical Institute, University of Oxford,
Oxford OX2\thinspace6GG, United Kingdom.}~\footnote{Research supported by ERC Horizon 2020 grant 883810.} \quad \;
Alex Scott\protect\footnotemark[1]~\footnote{Research supported by EPSRC grant EP/X013642/1.} \quad\;
Youri Tamitegama\protect\footnotemark[1] \quad\;
Jane Tan\protect\footnotemark[1]}
\date{}

\newtheorem*{rep@theorem}{\rep@title}
\newcommand{\newreptheorem}[2]{%
\newenvironment{rep#1}[1]{%
 \def\rep@title{#2 \ref{##1}}%
 \begin{rep@theorem}}%
 {\end{rep@theorem}}}
\makeatother

\newtheorem{theorem}{Theorem}
\newreptheorem{theorem}{Theorem}
\newtheorem{conjecture}[theorem]{Conjecture}
\newtheorem{corollary}[theorem]{Corollary}
\newtheorem{claim}{Claim}
\newtheorem*{claim*}{Claim}
\newtheorem{lemma}[theorem]{Lemma}
\newtheorem{fact}{Fact}

\newtheorem{proposition}[theorem]{Proposition}
\newtheorem{question}{Question}
\theoremstyle{definition}
\newtheorem{definition}{Definition}
\newtheorem*{definition*}{Definition}

\newenvironment{poc}{\begin{proof}}{\end{proof}}

\numberwithin{equation}{section}

\newcommand{\defn}[1]{\textcolor{Maroon}{\emph{#1}}}
\newcommand{\floor}[1]{\left\lfloor #1 \right\rfloor}

\newcommand{\R}{\mathbb{R}}
\newcommand{\thickening}{\mathbb{X}}
\newcommand{\thick}[1]{\mathbb{#1}}
\renewcommand{\leq}{\leqslant} 
\renewcommand{\geq}{\geqslant}

\DeclareMathOperator\sides{sides}
\newcommand{\cato}{{CAT}(0)}

\newcommand{\lk}{\operatorname{link}}
\newcommand{\interior}{\operatorname{int}}
\newcommand{\boundary}{\partial}

\newcommand{\eps}{\varepsilon}

\newcommand\blfootnote[1]{
  \begingroup
  \renewcommand\thefootnote{}\footnote{#1}
  \addtocounter{footnote}{-1}
  \endgroup
}
\advance\textwidth2\evensidemargin
\begin{document}
\maketitle
\begin{abstract}
The boundary rigidity problem is a classical question from Riemannian geometry: if $(M, g)$ is a Riemannian manifold with smooth boundary, is the geometry of $M$ determined up to isometry by the metric $d_g$ induced on the boundary $\partial M$?
In this paper, we consider a discrete version of this problem: can we determine the combinatorial type of a finite cube complex from its boundary distances?
As in the continuous case, reconstruction is not possible in general, but one expects a positive answer under suitable contractibility and non-positive curvature conditions.
Indeed, in two dimensions Haslegrave gave a positive answer to this question when the complex is a finite quadrangulation of the disc with no internal vertices of degree less than $4$.
We prove a $3$-dimensional generalisation of this result: the combinatorial type of a finite \cato\ cube complex with an embedding in $\R^3$ can be reconstructed from its boundary distances.
Additionally, we prove a direct strengthening of Haslegrave's result: the combinatorial type of any finite 2-dimensional \cato\ cube complex can be reconstructed from its boundary distances.

\blfootnote{Email: \texttt{\{\href{mailto:haslegrave@maths.ox.ac.uk}{haslegrave},\href{mailto:scott@maths.ox.ac.uk}{scott},\href{mailto:tamitegama@maths.ox.ac.uk}{tamitegama},\href{mailto:jane.tan@maths.ox.ac.uk}{jane.tan}\}@maths.ox.ac.uk}}
\end{abstract}

\section{Introduction}
The reconstruction of higher-dimensional structures from lower-dimensional information has been an important area of research for many years.  For example, the question of whether a Riemannian manifold with boundary is determined by its spectrum was popularized in a famous article of Mark Kac \cite{drum}; and there is a huge body of research on reconstructing discrete objects from their projections \cite{tomography}.

A particularly natural question of this type is whether the internal  structure of an object can be determined from distances between boundary points.  In Riemannian geometry,
the notion of reconstruction from a distance function on the boundary of a geometric object is well-established in the realm of boundary rigidity questions. Broadly, a Riemannian manifold $(M,g)$ is said to be \defn{boundary rigid} if its associated metric $d_g$ (which is defined on any two points, including the interior) is determined up to isometry by its boundary distance function given by the restriction $d_g|_{\boundary M\times \boundary M}$. In 1981, Michel~\cite{michel81} conjectured that every simple
compact Riemannian manifold with boundary is boundary rigid. The 2-dimensional case was verified by Pestov and Uhlmann~\cite{PU05}. In higher dimensions, however, the conjecture is wide open and has only been verified for a few classes~\cite{BCG95, BI10}.

There has been less work on analogous questions for discrete structures.  
Haslegrave, answering a question of Benjamini~\cite{benjamini}, proved the following result in two dimensions.
\begin{theorem}[Haslegrave \cite{john2D}]\label{thm:disk}
Let $Q$ be a planar quadrangulation with a simple closed boundary such that all internal vertices have degree at least $4$. Then the distances between boundary vertices determine $Q$ up to isomorphism.
\end{theorem}

Here, the distances are taken to be in the graph metric and the condition on the boundary of $Q$ can be restated by saying that $Q$ is (isomorphic to) a planar quadrangulation of a disc. Moreover, the degree condition is necessary for reconstruction.

Theorem~\ref{thm:disk} can be viewed as a discrete analogue of the 2-dimensional boundary rigidity result of  Pestov and Uhlmann~\cite{PU05}. The discrete case should be more approachable than the continuous one in general. This can be seen for instance in the fact that much stronger restrictions on the boundary are required in the latter.

In this paper, we look at generalising Theorem~\ref{thm:disk} to higher dimensions where the natural counterpart for a quadrangulation is a cube complex -- just as a quadrangulation can be formed by gluing Euclidean squares (or 2-dimensional cubes) along edges, a $k$-dimensional cube complex is informally a complex formed by gluing together cubes of dimension at most $k$ along subcubes. This leads to the following question.

\begin{question}\label{qu:main}
Under what conditions is a finite $k$-dimensional cube complex $X\hookrightarrow \R^k$ reconstructible up to combinatorial type from its boundary distances?
\end{question}

Question~\ref{qu:main} requires not only determining the full $1$-dimensional structure from boundary distances alone, but also reconstructing the higher-dimensional structure from the $1$-dimensional. This is not really true of Theorem~\ref{thm:disk}: while the goal is to reconstruct a $2$-dimensional complex, any polyhedral graph has a unique embedding in the sphere~\cite{whitney}, and hence a unique embedding in the plane with a designated outer face, so the second step is immediate in this case.

We provide an answer to Question~\ref{qu:main} for $3$-dimensional complexes, with a well-studied condition which directly generalises the one stated in Theorem~\ref{thm:disk}.
We require complexes to satisfy the \cato\ property, which entails both a global topological condition (simply connectedness) and a local negative curvature condition (Gromov's link condition, which states that the link of every vertex is a flag complex). 
Analogously to \cato\ spaces, \cato\ cube complexes form a large, popular class of complexes possessing useful convexity properties (see Section~\ref{sec:toolbox}).
This makes them a natural choice of setting for boundary rigidity problems.
Our main theorem is the following.

\begin{theorem}\label{thm:main}
Suppose that $X$ is a finite \cato\ cube complex admitting an embedding in $\R^3$, with a labelling of vertices in $\partial X$.
Let $D$ be the matrix of pairwise distances between vertices of $\partial X$ with respect to the graph metric on the $1$-skeleton of $X$.
Then the combinatorial type of $X$ is reconstructible from $D$.
\end{theorem}
In fact, we preserve the labelling of boundary vertices when reconstructing the combinatorial type.

Both the simply connectedness and flag conditions are used essentially in numerous places throughout the proof. It is also possible to see directly that they cannot be omitted from the statement. For instance, if we do not require links to be flag, one could `hide' a cube inside another, as in Figure~\ref{fig:hiddencube}: geodesics between vertices on the outer cube are unaffected by the presence of the inner cube.
This is a $3$-dimensional analogue of `hiding' a square within another in a quadrangulation of the disc.
Another example is given by taking a $3\times3\times3$ block of cubes and considering the cube complexes formed by removing the top two cubes in the centre column and by removing the top and bottom cubes in the centre column (depicted in Figure~\ref{fig:needflag}). These two cube complexes do not satisfy the flag condition at any vertex of the (possibly missing) middle cube, and it is easily seen that they have the same boundary distances since all vertices are on the boundary and both complexes have the same edges. Similarly, we can see that contractibility, which implies simply connectedness, is necessary as it would be impossible to differentiate between a single square with or without a face only from the boundary distances.

 \begin{figure}[t]
    \centering
    \includegraphics[scale=0.4]{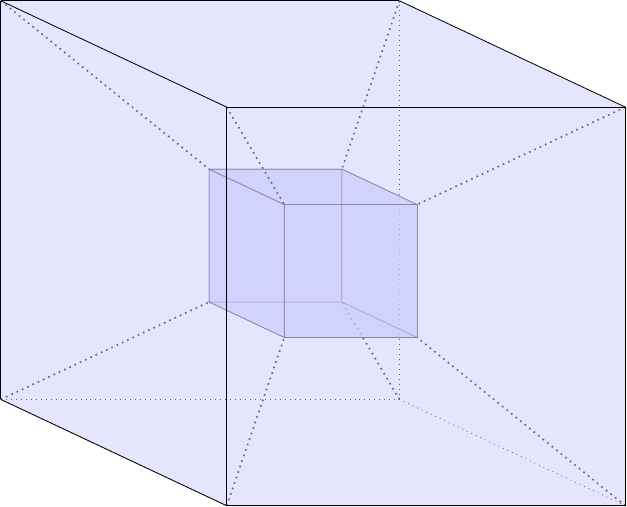}
    \caption{A cube hidden within another. Links of vertices of the `hidden' cube are not flag: four cubes meet at each such vertex, but the complex has no cells of dimension $4$.}
    \label{fig:hiddencube}
\end{figure}
 \begin{figure}[t]
    \centering
    \includegraphics[scale=0.4]{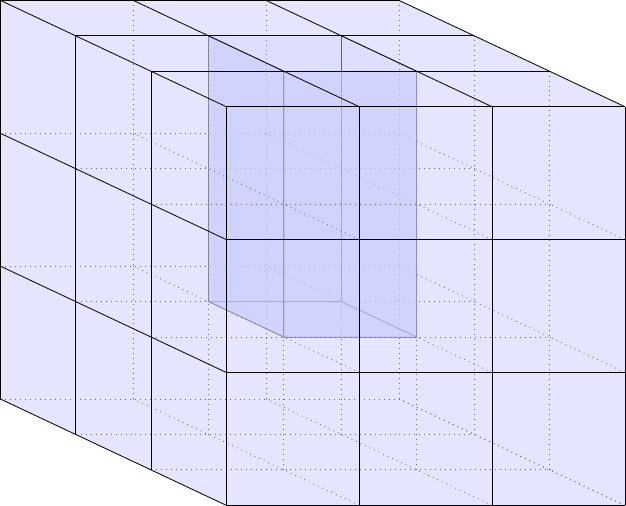}\hspace{1.6cm}
    \includegraphics[scale=0.4]{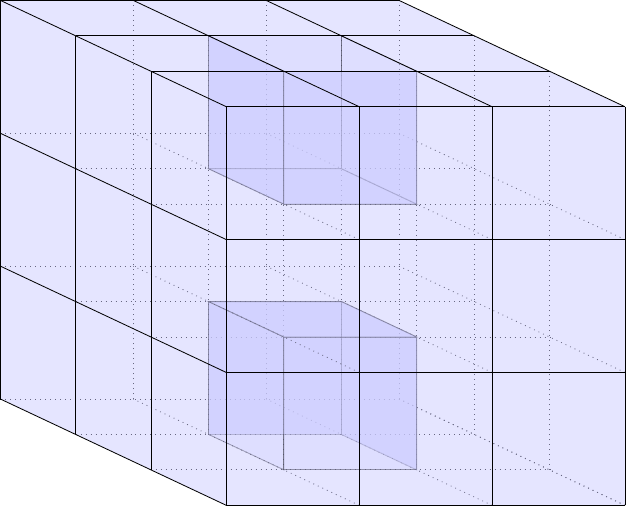}
    \caption{Two pure cube complexes with the same boundary distances. The dark shaded cubes are removed.}
    \label{fig:needflag}
\end{figure}

 \cato\ cube complexes are ubiquitous in modern geometric group theory. While the importance of the condition in the context of reconstruction may not be immediately clear, it does in fact directly generalise earlier conditions. To see this, we note that Gromov's link condition reduces to the degree condition of Theorem~\ref{thm:disk} for cube complexes of top dimension at most $2$, and local negative curvature is also one of the key assumptions used by Besson, Courtois and Gallot in the continuous setting~\cite{BCG95}. Furthermore, the fact that $2$-dimensional cube complexes in Theorem~\ref{thm:disk} are contractible is also captured in the \cato\ condition, which follows from the Cartan--Hadamard theorem. In fact \cato\ complexes have the even stronger property of collapsibility (see~\cite{adiprasito2019collapsibility}).

Question~\ref{qu:main} asks whether one can recover combinatorial information of a cube complex $X$ from some partial combinatorial information, namely the distance (in the graph of the entire complex) between any two vertices on the boundary.
In order for this question to be well-defined, we first need to make precise the notion of \emph{boundary}. 
In the present paper we mostly work with the natural notion of \defn{geometric boundary}: given an embedding $X\hookrightarrow \R^k$, we define $\partial X$ to be the topological (induced from the Euclidean metric) boundary of $X$. 
This notion implicitly depends on the dimension $k$ of the space that we embed $X$ into.
Indeed, if $X$ has no cells of dimension at least $k$ then it is its own boundary in any embedding $X\hookrightarrow \R^k$, so reconstructing $X$ from boundary information is trivial.
Hence, we always consider $k$-dimensional cube complexes embedded in $k$-dimensional Euclidean space.

A convenient observation is that for finite cube complexes the geometric boundary is independent of the embedding we choose, so long as one exists. With this in mind, one can define a combinatorial notion of boundary for cube complexes. This is discussed in more detail in Section~\ref{sec:defns} and this notion of boundary will be used in Section~\ref{sec:dim12}.

Since this combinatorial notion of boundary does not require an embedding, it suggests a natural generalisation of Question~\ref{qu:main}: can we reconstruct \cato\ cube complexes which do not necessarily admit embeddings in a euclidean space of their top dimension?
We give a positive answer for cube complexes of top dimension at most $2$, providing in particular a strengthening of Theorem~\ref{thm:disk}.
\begin{theorem}\label{thm:dim12}
Let $X$ be a \cato\ cube complex of top dimension at most $2$ with finitely many cells and $D$ its matrix of pairwise distances between vertices on the combinatorial boundary of $X$.
Then, the combinatorial type of $X$ is reconstructible from $D$.
\end{theorem}

This work makes significant use of notions and tools from algebraic topology and from the theory of \cato\ cube complexes.
Since there is a good deal of terminology involved, we postpone technical discussions and first introduce notation and necessary theory in Section~\ref{sec:defns}. 
With this background in hand, we give a skeletal version of the proof of Theorem~\ref{thm:main} in Section~\ref{sec:outline}. 
This provides a roadmap for Sections~\ref{sec:cleaning} through \ref{sec:rowsonboundary}, which are devoted to the different aspects of the main proof, while Section~\ref{sec:toolbox} collects and proves basic technical lemmas needed to make our arguments rigorous. The reader may wish to skip this section at first and use it as a reference. Finally, while this paper lays the groundwork for $k$-dimensional generalisations ($k\geq 4$), there are non-trivial complications that arise. We briefly discuss this together with other open questions in Section~\ref{sec:conclusion}.
In Section~\ref{sec:dim12} we provide a brief sketch of the proof of Theorem~\ref{thm:dim12}.

\section{Definitions}
\label{sec:defns}

In this section, we define the key objects and terminology that we will be working with. 
This will then allow us to give a broad outline of the proof of \cref{thm:main}.

We start with some standard topological notions. 
For simplicity of exposition, we will restrict some definitions to the cases that we require, although they may exist in much greater generality. 
We refer to \cite{hatcher} for a detailed account of the concepts from point-set and algebraic topology and the basic definition of CW complexes, and \cite{wise} for specifics on cube complexes. 
All CW complexes we consider are \defn{regular}, meaning that their gluing maps are homeomorphisms.
Our notation and descriptions below are chosen to reflect the fact that we will require a mix of combinatorial and geometric properties of the objects in question. 

We write $I$ for the unit interval $[0,1]$. Let $S^k$ and $B^k$ be the unit sphere and ball in $\R^k$ respectively, i.e.
\begin{align*}
S^k &= \{(x_1,x_2,\dotsc,x_k)\in \R^k: x_1^2+\dotsc +x_k^2=1\},\\
B^k &= \{(x_1,x_2,\dotsc,x_k)\in \R^k: x_1^2+\dotsc +x_k^2\leq 1\}.
\end{align*}

\subsection{Geometric boundary}
Let $X$ be a regular CW complex whose gluing maps are isometries.
Given an embedding $X\hookrightarrow \R^n$, the \defn{(geometric) boundary} of $X$, denoted $\boundary{X}$, is the set of points in $X$ for which every neighbourhood intersects both $X$ and $\R^n\setminus X$.
The \defn{interior} of $X$ is then $\interior{X}\coloneqq X\setminus\boundary{X}$. 
Explicitly, this is the set of points $p\in X$ such that the ball $B_\eps(p)$ is contained in $X$ for some sufficiently small $\eps>0$, where $B_\varepsilon(p)$ denotes the \defn{ball of radius $\varepsilon$ centered at $p$.}

If finiteness of the complex is not assumed, this notion of boundary may depend on the chosen embedding.
For example, consider the embedding of a $2$-dimensional complex $Z$ in $D\subseteq \mathbb{C}$ obtained by gluing $2$-cells to the sectors bounded by consecutive vectors (viewed as edges) from $\{ e^{i\pi\sum_{j=0}^k 2^{-j}} \colon k\in \mathbb{N} \}$.
The geometric boundary from this embedding is the preimage of the unit circle, so in particular does not contain the (preimage of the) edge $e^{i\pi}$.
Yet, there are ways of embedding $Z$ in $\mathbb{C}$ where the preimage of this edge is on the geometric boundary: e.g. by gluing $2$-cells to the sectors bounded by consecutive vectors from $\{ e^{i\pi\sum_{j=1}^k 2^{-j}} \colon k\in \mathbb{N} \}$.

For a regular CW complex $X$ of maximum dimension $k$ whose gluing maps are isometries, we define its \defn{combinatorial boundary} to be the downward closure of the cells of dimension less than $k$ in at most one cell of dimension $k$. 
In general, the geometric and combinatorial notions of boundary are different even for complexes admitting embeddings in $\R^k$, as the combinatorial boundary is independent of any embedding.
For finite complexes however, these two notions coincide, so long as there exists an embedding in $\mathbb{R}^k$. 

\begin{proposition}
Let $X$ be a regular $k$-dimensional CW complex whose gluing maps are isometries and which admits an embedding $X\hookrightarrow \R^k$.
If $X$ has a finite number of cells, then its combinatorial and geometric boundaries are the same.
\end{proposition}
\begin{proof}
Let $X$ be a finite regular CW complex of dimension $k$ and fix an embedding in $\R^k$.
Suppose that a cell $S$ of dimension less than $k$ is contained in at most one cell of dimension $k$. Suppose that a point $x$ inside $S$ (that is, in $S$ but not in any lower dimension cell) is not in the geometric boundary. Since the complex is finite, there is some minimum distance between $x$ and the union of all $k$-cells not containing $S$. Take a ball of radius smaller than this. Now any point in the ball that does not intersect the $(k-1)$-skeleton must be in some $k$-cell. Furthermore, there must be at least two such cells involved, since if there is a unique such cell $T$ then the whole ball is in (the closure of) $T$, and so, since gluing maps are isometries, $x$ is in the interior of $T$, contradicting the choice of $x$. So all of the interior of $S$ is on the geometric boundary. Since the geometric boundary is closed, it also contains all cells of $S$, so the geometric boundary contains the combinatorial boundary.

Now consider a point $x$ that lies inside the combinatorial boundary.
Suppose $S$ is a $(k-1)$-cell contained in two $k$-cells, and let $x$ be a point in the interior of $S$. Then a sufficiently small ball around $x$ meets no cells other than these three. 
By passing to a smaller ball if necessary we can assume that it is divided into two parts by $S$, either of which contains interior points of the larger cells.
Thus each part of the ball is contained in one of the two cells, and $x$ is not on the geometric boundary.

Suppose $x\in S$ is in the geometric boundary, where $S$ is a cell of dimension at most $k-2$ and is the inclusion minimal cell containing $x$. If $S$ is not in the combinatorial boundary, then every $(k-1)$-cell containing $S$, of which there is at least one, is in two $k$-cells. Take a ball around $x$ that is sufficiently small to avoid any $(k-1)$-cell not containing $S$. This ball contains a point $y$ in the interior of some $k$-cell (since $x$ lies in the closure of such a cell), and a point $z$ outside the complex (since $x$ is in the geometric boundary). Now the ball is path-connected, even if we remove the $(k-2)$-skeleton from it. Thus there is a path from $y$ to $z$, which avoids the $(k-2)$-skeleton and must contain a point on the geometric boundary. By the previous paragraph, no point in the interior of a $(k-1)$-cell containing $S$ is in the geometric boundary, a contradiction.
\end{proof}
In particular, this result applies to the cube complexes considered throughout this paper.
Henceforth, all CW complexes are assumed to have a finite number of cells.

\subsection{Simplicial complexes}
An \defn{$n$-simplex} is an $n$-dimensional object formed by taking the convex hull of $n$ linearly independent vectors. 
Every $n$-simplex is homeomorphic to a \defn{standard $n$-simplex} $\Delta^n := \{(x_0,\ldots,x_n) \subseteq \R^{n+1} : \sum_i x_i=1 \text{ and }x_i\geq 0\text{ for all }i\}$, which is spanned by the unit vectors along each coordinate axis.
We say that $x_0,\dotsc, x_n$ span the simplex $\Delta^n$.
Low-dimensional simplices are familiar objects: we will call $0$-simplices \defn{vertices}, 1-simplices \defn{edges}, 2-simplices \defn{triangles} and 3-simplices \defn{tetrahedra}. 

Recall that a \defn{simplicial complex} $S$ is a CW complex whose cells is a collection of simplices such that
\begin{itemize}[itemsep=-1mm]
    \item for every simplex in $S$, all of its simplicial faces are also in $S$, and
    \item the intersection of any two simplices in $S$ is a simplicial face of both of them.
\end{itemize}
The \defn{dimension} of a simplicial complex is the dimension of its top-dimensional simplices.
We say that a simplicial complex $S$ is \defn{flag} if whenever there is a collection of $k$ pairwise adjacent vertices (that is, joined by edges), then those $k$ vertices span a $(k-1)$-simplex in $S$. 
Informally, this means that there is a $k$-simplex everywhere there should be one according to the graph of vertices and edges in the complex. 
Finally, note that the boundary of an $n$-dimensional simplicial complex has a natural structure as an $(n-1)$-dimensional simplicial complex.

\subsection{Cube complexes}
We now turn to \defn{cube complexes}, which are CW complexes whose $n$-cells are $n$-cubes and gluing maps are combinatorial isometries.
Let the \defn{standard $n$-cube} be the set of points $Q^n\coloneqq I^n \subseteq \R^n$ (the standard 0-cube is a single point). 
Note that each $n$-cube is endowed with a natural internal coordinate system. 
By restricting any $k$ of the coordinates to 0 or 1, we obtain an $(n-k)$-cube on the boundary of our $n$-cube which we call a \defn{cubical face}. In general, an $n$-cube is any set in $\R^n$ that is homeomorphic to the standard $n$-cube. 
The \defn{dimension} of a cube complex is defined analogously to that for simplicial complexes. An $n$-dimensional cube complex is \defn{pure} if every $k$-cube with $k<n$ is contained in an $n$-cube. 
Since we mainly work with $3$-dimensional cube complexes, in the later sections of this paper we will use the terms \defn{vertices}, \defn{edges}, \defn{faces}, and \defn{cubes} (with no specified dimension) to mean $0$-cubes, $1$-cubes, $2$-cubes and $3$-cubes respectively. 
As such, we will refer to the $1$-skeleton $X^1$ as the \defn{graph} of $X$, with the corresponding graph theoretic terminology.
In particular, unless otherwise stated, a \defn{path} $\gamma$ in $X$ is a graph path in its graph $X^1$ and its \defn{length} $|\gamma|$ is its number of edges, and a \defn{walk} is a path where vertices may be repeated.
When $\gamma$ consists of vertices $v_0,\dotsc, v_k$, in this order, we sometimes use the notation $v_0\dotsb v_k$ for $\gamma$.

For $k,\ell\geq 0$ we say that an $\ell$-cube in a $k$-dimensional cube complex is \defn{free} if it is not contained in any $k$-cube, this being one possible structure in a cube complex that is not pure. 
We will mainly use this term to refer to free faces in $3$-dimensional cube complexes, i.e. $2$-cubes not contained in $3$-cubes.
A \defn{cubulation of the ball} is a cube complex that is homeomorphic to $B^3$. 
In Section~\ref{sec:thethickening} we will see a construction which requires fixing an embedding $X\hookrightarrow \R^3$; 
as such, when introducing a cube complex $X$ we use this notation to indicate that we have fixed a specific embedding of $X$.
Say that a vertex is the \defn{corner of a cube (face)} if all cells containing it are contained in a unique cube (face). 

In this paper we are interested in the \defn{graph metric} on the $1$-skeleton of cube complexes, meaning the length of shortest paths between vertices.
When it is unambiguous to do so we will talk about geodesics on cube complexes to mean geodesics on their $1$-skeleton.
As such, if $X$ is a cube complex, the \defn{$1$-distance graph induced by the boundary} $G(X)$ is the subgraph of $X^1$ induced by the vertices $\partial X^0$.
This contains $\partial X^1$ as a subgraph, which may be a proper subgraph since edges not in $\partial X$ can have endpoints on the boundary.
Hence, we can read $G(X)$ off the restriction $D_{\partial X}$ of the distance matrix to the boundary, but not necessarily $\partial X^1$.

A map $f\colon X\rightarrow Y$ between CW complexes $X$ and $Y$ is said to be \defn{combinatorial} if its restriction to the boundary of any cell of $X$ is injective, and if it maps each cell to a cell of same dimension.
Two cube complexes $X$ and $Y$ are said to have the same \defn{combinatorial type} if there are bijections $f_i\colon X^i\rightarrow Y^i$ for each dimension $i$ such that any two cells $\sigma$, $\tau$ of $X$ are incident if and only if $f(\sigma)$ is incident to $f(\tau)$ in $Y$.
In this paper, we will mostly be interested in cube complexes up to combinatorial type, meaning that we consider them to be \emph{distinct} when their combinatorial types differ.

There is an important construction which allows us to encode local structural information from a cube complex via an auxiliary simplicial complex. Given a cube complex $X$ and a vertex $v\in X$, the \defn{link} of $v$, denoted $\lk(v)$, is the simplicial complex where:
\begin{itemize}
    \item the vertices of $\lk(v)$ are in bijection with edges containing $v$,
    \item for $n\geq 2$, there is an $(n-1)$-simplex with vertices $e_1,\ldots, e_{n}$ in $\lk(v)$ if and only if there is an $n$-cell $C$ in $X$ containing $v$ where $e_1,\ldots, e_{n}$ are the edges of $C$ that contain $v$. 
\end{itemize}
Intuitively, the simplices in $\lk(v)$ correspond to `corners' of cells in $X$ that contain $v$. A useful alternative perspective, assuming that $X$ is finite and embedded in Euclidean space, is that $\lk(v)$ is the intersection of the sphere $S_\eps(v)$ with $X$ for sufficiently small $\eps>0$. This has a natural simplicial structure.

\subsection{\cato\ Cube complexes}
At last, we arrive at the key property that we need for reconstruction. A cube complex $X$ is \defn{\cato} if it is simply connected and $\lk(v)$ is flag for every $v\in X^0$. The latter part of this definition is really a condition requiring that the complex has nonpositive curvature, and in fact directly generalises the degree condition in Theorem~\ref{thm:disk}. 
In a disc quadrangulation, the link of each boundary vertex is a path, while the link of an internal vertex is a cycle. 
Thus a disc quadrangulation is \cato\ if and only if cycles in links have length at least $4$, i.e.\ if and only if each internal vertex has degree at least $4$. However, in three dimensions there is no corresponding equivalence: being \cato\ implies that every internal vertex has degree at least $6$, but the flag condition may fail at an internal vertex even if it has high degree, and it may also fail at a boundary vertex.
The structure of links of vertices in \cato\ cube complexes will be crucial for us.
We continue this discussion in Section~\ref{sec:links}.

While there is a great deal of rich theory surrounding \cato\ cube complexes -- especially concerning their applications in geometric group theory, we will only need basic combinatorial considerations for our purposes and will refrain from delving deeper in the existing theory. 
Two important objects of study will be \emph{immersed hyperplanes} and \emph{disc diagrams}.

\subsection{Hyperplanes and disc diagrams}
For $n\geq 1$, a \defn{midcube} in an $n$-cube $C=I^n$ is a codimension $1$ cube $M$ with an embedding $M= I^{n-1}\times \{1/2\}$ in $C$.
As such, $C$ has precisely $n$ distinct midcubes, and the intersection of each midcube with a face of $C$ of codimension at least $2$ is again a midcube of that face.
Moreover, two midcubes of distinct cubes in a cube complex $X$ intersect in a combinatorial manner, meaning that the natural gluing map between the midcubes is combinatorial.
In this way, midcubes of cubes of dimension at least $1$ form connected components which we refer to as \defn{hyperplanes}.
We emphasize the distinction between a hyperplane $H$ as a standalone cube complex and its natural embedding $H\rightarrow X$ as midcubes by referring to the latter as an \defn{immersed hyperplane}.
Then, the \defn{(cubical) neighbourhood} of $H\rightarrow X$, written $N(H)$ or $H\times I$ (sometimes called the \emph{carrier} of $H$), is the union of cubes in $X$ containing it.
We say that two immersed hyperplanes \defn{cross} (in a complex $X$) if they contain two midcubes of some cube of any dimension (in this complex $X$); when a hyperplane crosses itself we say that it \defn{self-crosses}.

A \defn{disc diagram} is a locally injective combinatorial map $D\rightarrow X$, where $D$ is a quadrangulation of the disc. 
We now define some pathological substructures in disc diagrams $D\rightarrow X$.
In this setting, hyperplanes are $1$-dimensional cube complexes, i.e. graphs.
A \defn{nonogon} is the neighbourhood of an immersed hyperplane which is a cycle.
Bigons can be defined from two non self-crossing immersed hyperplanes crossing each other at least twice in $D$: we define a \defn{bigon} to be the cubical neighbourhood of two paths in such hyperplanes crossing each other exactly twice which are inclusion minimal with this property.
Notice that these definitions also apply for \cato~cube complexes of dimension at most $2$.

 \begin{figure}[t]
    \centering
    \begin{subfigure}{0.2\textwidth}
    \centering
    \includegraphics[scale=0.4]{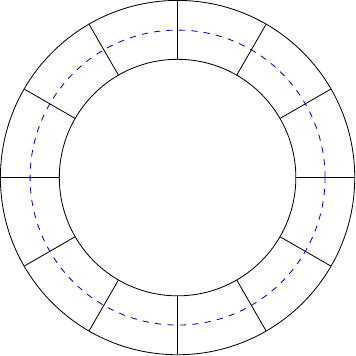}\hspace{1.6cm}
    \end{subfigure}
    \begin{subfigure}{0.2\textwidth}
    \centering
      \raisebox{0.5\height}{\includegraphics[scale=0.4]{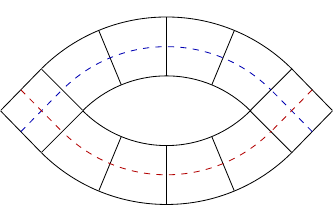}}
    \end{subfigure}
    \begin{subfigure}{0.2\textwidth}
    \centering
\raisebox{0.5\height}{\includegraphics[scale=0.4]{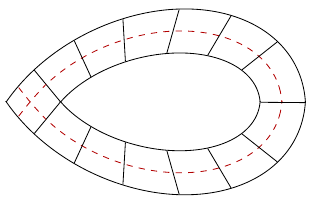}}
    \end{subfigure}
    \caption{From left to right: a nonogon, a bigon and a self-intersecting hyperplane. 
    The relevant immersed hyperplanes are represented in dashed lines.}
    \label{fig:nonogon}
\end{figure}

Both cells and simplices are specified by their vertices, so we will refer to a particular within a complex by a set of vertices. In addition, if $X$ is a simplicial complex (or cube complex), let $X^i$ be the \defn{$i$-skeleton} of $X$ which is the union of all $k$-simplices ($k$-cells) for $0\leq k \leq i$.

In this setting, a \defn{minimal disc diagram} is understood to be, for a fixed cycle $\gamma$ in $X^1$, a disc diagram $D\hookrightarrow X$ whose boundary is mapped to $\gamma$, chosen so that it minimises the number of faces, edges and vertices.
Hyperplanes and minimal disc diagrams are particularly well-behaved in \cato\ cube complexes, a statement which we make precise in Section~\ref{sec:hyperplanesbasic}.

\section{Proof overview and discussion}
\label{sec:outline}

This section provides, in a skeletal form, the proof of our main theorem. The main purpose is to provide a break-down of the proof into the components that span the remaining sections of this paper, as well as discuss the necessity of certain approaches.

As discussed in the introduction, the conditions in Theorem~\ref{thm:main} are necessary.
To prove sufficiency we proceed by induction, beginning with complexes with at most one edge. At each step, we aim to reduce the size of our cube complex: we have four processes that each use a particular substructure within the complex to define one or more smaller subcomplexes on which the induction hypothesis can be applied. This entails a number of verifications, namely that:
\begin{enumerate}
\item each substructure can be recognised from the boundary distance data that we start with, 
\item all resulting subcomplexes still satisfy the \cato\ property,
\item the boundary distances in all resulting subcomplexes can be recovered, and
\item if we are not in the base case, then at least one of the substructures exists in our finite \cato\ cube complex so that a reduction can be performed. 
\end{enumerate}

Our four chosen substructures are cut-vertices, corners of faces, vertices of degree $3$ that are not in any cube, and rows of cubes on the boundary. 
For the first three of these structures, there is a natural way to reduce our complex into smaller pieces and the corresponding verifications are relatively straightforward. These are detailed in Section~\ref{sec:cleaning}, and allow us to proceed with the assumption that our complex $X$ does not contain any of these three structures in which case we call $X$ `clean'.

The main work in our proof lies in handling rows of cubes on the boundary -- essentially maximal stacks of cubes with one side on the boundary (the precise definition is given in Section~\ref{sec:rowsofcubes}). These structures are a natural choice for induction arguments in \cato\ cube complexes because of their well-behaved hyperplanes. 
Here they work nicely in that they can easily be read off the boundary distance matrix, and their removal (for several natural definitions of removal) leave a subcomplex where the flag condition is preserved at each vertex and their hyperplanes allow us to recover boundary distances to newly created boundary vertices. 
These properties are proved in Section~\ref{sec:rowsofcubes}. Within our proof, rows of cubes are key to making bulletpoint 4 above true. The intuition for this comes from the controlled case when $X$ is a finite \cato\ cube complex that is homeomorphic to a ball, where a simple Euler characteristic argument can be used to show that there must exist a row of cubes on the boundary of $X$. 

Unfortunately, the property of being homeomorphic to a ball is not necessarily preserved when removing rows of cubes. See Figure~\ref{fig:Lshape} for instance: removing the central cube from three cubes glued together to form an `L' shape.
One could hope to reduce the resulting complexes by `splitting' then appropriately into subcomplexes homeomorphic to balls, but this approach is complicated by the fact that \cato\ cube complexes may possess `essential' lower dimensional features, in the sense that removing these yields complexes with non-trivial homotopy. See for example Figure~\ref{fig:ringeg}.

\begin{figure}[t]
    \centering
    \centering
    \includegraphics[scale=0.6]{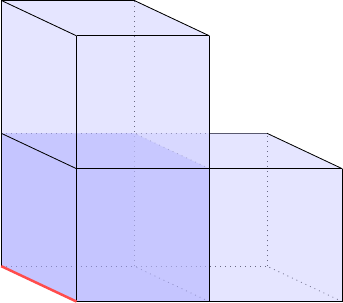}
    \caption{\cato\ cube complex where removing a row of cubes gives a complex which is not homeomorphic to a ball: the dark blue cube is a row of cubes on the boundary (its red edge is a path of length $1$ with endpoints and internal vertices of boundary degrees $3$ and $4$ respectively), yet removing it leaves two cubes sharing an edge.}
    \label{fig:Lshape}
\end{figure}
\begin{figure}[t]
    \centering
    \centering
    \includegraphics[scale=0.45]{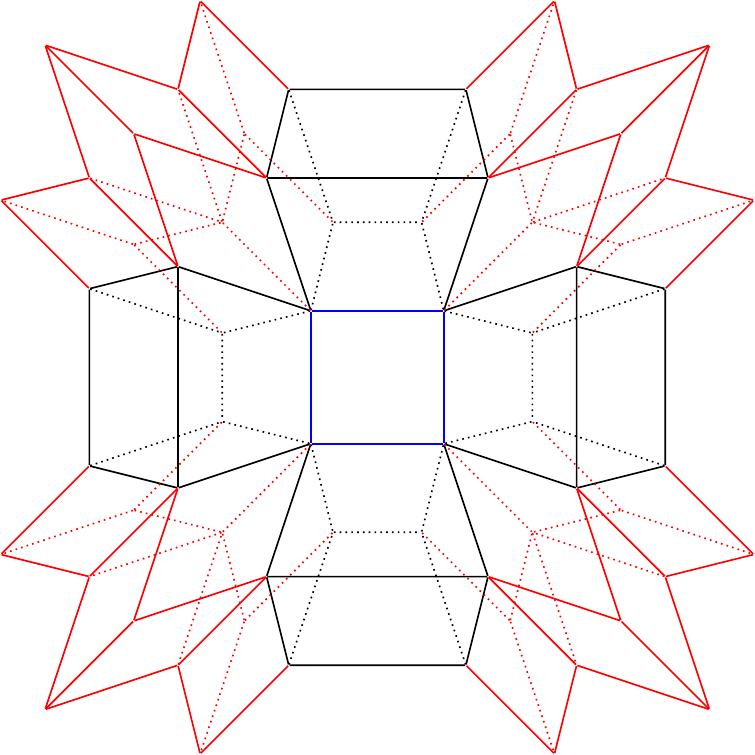}
    \caption{\cato\ cube complex where removing the face bounded by the central blue square gives a complex which is not contractible.}
    \label{fig:ringeg}
\end{figure}

To overcome this, we use a removal process where we leave the `back wall' of the row of cubes intact, thus ensuring contractibility. 
However this forces us to adapt our arguments for \cato\ complexes which may not have fixed Euler characteristic. Our approach is to `thicken' the complex $X$. This procedure, which is the topic of Section~\ref{sec:thethickening}, produces a cube complex $\mathbb{X}$ homeomorphic to the ball $B^3$ which contains $X$. Roughly speaking, this is achieved by taking $X$ together with a cubical shell around $X$.
This shell is constructed in such a way that there is a correspondence between the graphs of $\partial \thick{X}$ and $\partial X$ (Lemma~\ref{lem:thicksctooriginalsc}).
In particular, this correspondence allows us to transfer the previously mentioned Euler characteristic argument to the thickening $\thick{X}$.

\begin{proof}[Proof of Theorem~\ref{thm:main}]
Let $X$ be a contractible \cato\ cube complex with an embedding in $\mathbb{R}^3$.
We proceed by induction on the number of vertices of $\partial X$. We may assume without loss of generality that $X$ is clean, as otherwise, by the discussion described in Section~\ref{sec:cleaning} we can recognise this and perform a reduction to reduce to a complex with fewer vertices on the boundary which satisfies the induction hypothesis.

By Lemma~\ref{lem:goodscshell}, the thickening $\thick{X}$ of $X$ admits a good configuration $\mathcal{C}$. 
Under the correspondence described in Lemma~\ref{lem:thicksctooriginalsc}, $\mathcal{C}$ corresponds to a good configuration $\pi(\mathcal{C})$ in $X$. 
This good configuration is a pattern in the boundary distance matrix which can be recognised, and by Lemma~\ref{lem:scrowcorrespondance}, $\pi(\mathcal{C})$ corresponds in turn to a row $R$ of cubes on the boundary of $X$.
Finally, applying the reduction described in Section~\ref{sec:removingrowofcubes}, we reduce to a complex with smaller boundary which satisfies the induction hypothesis.
\end{proof}

\section{Technical toolbox}\label{sec:toolbox}
\subsection{Hyperplanes in \cato\ cube complexes}\label{sec:hyperplanesbasic}
Let $X$ be a cube complex and $Y\subseteq X$ a subcomplex.
Recall that the metric we consider is the graph metric on $1$-skeleta.
In this setting we say that $Y$ is \defn{convex} if any geodesic with both endpoints in $Y$ is entirely contained in $Y$.

A consequence of the \cato\ property in cube complexes is the presence of natural convex subcomplexes, namely neighbourhoods of immersed hyperplanes.
This convexity is crucial to recover boundary distances after removing parts of a \cato\ cube complex, and we use it in the form of the following theorem, based on a paper of Sageev \cite{sageev1995ends}.
\begin{theorem}\label{thm:wise}[Wise~\cite[Theorem~2.13]{wise}]
Let $X$ be a $CAT(0)$ cube complex.
\begin{enumerate}[label=(\roman*), font=\textup]
\item \label{thm:wise1} Each midcube lies in a unique immersed hyperplane.
\item \label{thm:wise2} Hyperplanes are $CAT(0)$ cube complexes.
\item \label{thm:wise3} The cubical neighbourhood $N(H) = H\times [0,1]$ of an immersed hyperplane $H$ is a convex subcomplex.
\item \label{thm:wise4} $X\setminus H$ consists of two connected components.
\end{enumerate}
\end{theorem}

\subsection{Links of vertices}\label{sec:links}
The links of vertices in a complex contain important local information.
We will mostly be interested in deducing information about the local structure around vertices from their degree and existing partial information.

\begin{lemma}\label{lem:basic}
Let $X$ be a finite \cato\ cube complex with an embedding in $\R^3$ and $v$ a vertex of $\partial X$.
Then:
\begin{enumerate}[label=(\alph*), font=\textup]
\item \label{lem:basic1} For any subcomplex $Y\subseteq X$ containing $v$, there is a natural containment map $\lk_{Y}(v) \allowbreak \hookrightarrow \lk_{X}(v)$.
In particular, $\lk_{\partial X}(v)$ has at least as many components as $\lk_{X}(v)$.
\item \label{lem:basic2} If $v$ is contained in a cube $C$ of $X$, then $\lk_{\partial X}(v)$ has at least $3$ vertices.
If $\lk_{\partial X}(v)$ is further a single cycle, then $\lk_X(v)$ is homeomorphic to a disk $D^2$.
\item \label{lem:basic3} If $\lk_{\partial X}(v)$ does not contain a cycle, then the containment $\lk_{\partial X}(v) \subseteq \lk_{X}(v)$ is a bijection.
\item \label{lem:basic4} Suppose $\lk_{\partial X}(v)$ is a single cycle and $H\subseteq (\lk_{X}(v))^1$ is a subgraph.
This inclusion corresponds to an embedding of $H$ in an $\varepsilon$-sphere about $v$ in $\mathbb{R}^3$ which is a planar drawing for $H$.
Suppose that the following hold:
\begin{itemize}[noitemsep]
\item $H$ is a triangulation of $\lk_{\partial X}(v)$, meaning that $(\lk_{\partial X}(v))^1\subseteq H$ and this natural containment is a planar drawing of $H$ such that all vertices of $\lk_{\partial X}(v)$ lie on the outer face, and every other face is a triangle.
\item The following diagram, where the maps are the aforementioned natural containments, commutes. 
\[
  \begin{tikzcd}
     \lk_{\partial X}(v) \arrow{d} \arrow{r} & \arrow{dl}H & \quad \\
     \lk_{X}(v) & & \quad
  \end{tikzcd}
\]
\end{itemize}
Then the above embedding describes an isomorphism between $H$ and $(\lk_X (v) )^1$.
\end{enumerate}
\end{lemma}
A useful consequence of the second bullet point is that when $X$ is homeomorphic to the ball $B^3$, the link of any of its boundary vertices is homeomorphic to a disc $D^2$.

\begin{proof}
\begin{enumerate}[label=(\alph*), font=\textup]
\item The embedding $Y \hookrightarrow X$ gives a natural embedding $\lk_{Y}(v) \hookrightarrow \lk_{X}(v)$ through the identification between $k$-cells incident to $v$ and $(k-1)$-cells in the links at $v$.
In particular, the natural embedding $\partial X\hookrightarrow X$ guarantees that $\lk_{\partial X}(v)$ has at least as many connected components as $\lk_{X}(v)$: a path $f$ (interpreted here as a continuous map $f\colon [0,1] \rightarrow \lk_{\partial X}(v)$) in $\lk_{\partial X}(v)$ extends to a path $f'\colon[0,1]\rightarrow \lk_{X}(v)$ by precomposing with the inclusion map $\lk_{\partial X}(v) \hookrightarrow \lk_X(v)$ as the former is obtained from the latter by removing simplices.

\item Consider an $\varepsilon$-sphere around $v$ in $X$.
The connected component of $\lk_X(v)$ containing the $2$-simplex corresponding to $C$ is $2$-dimensional and thus -- as $\lk_X(v)$ is flag and therefore has no double edges -- has at least three vertices on its boundary, as desired.

If $\lk_{\partial X}(v)$ is a single cycle, $\lk_{\partial X}(v)\subseteq \lk_{X}(v)$ is homeomorphic to $S^1$ and so separates the $\varepsilon$-sphere into two components homeomorphic to discs $D^2$.
One of these components is $\lk_{X}(v)$, as desired.

\item If $\lk_{\partial X}(v)$ does not contain a cycle, then an $\varepsilon$-sphere about $v$ is not disconnected by removing $\lk_{\partial X}(v)$. It follows that $\lk_{X}(v)$ lies entirely on the boundary.

\item First, $\lk_{\partial X}(v)$ corresponds to a cycle in a $\varepsilon$-sphere about $v$, separating this sphere into two parts.
Since $H$ is a triangulation and $\lk_{X}(v)$ is flag, $\lk_{X}(v)$ contains at least one $2$-simplex and thus (exactly) one of the two parts of the $\varepsilon$-sphere is contained in $\lk_{X}(v)$.

Since links of $X$ are flag and $H\subseteq (\lk_X(v))^1$ is a triangulation, each triangle of the planar drawing of $H$ bounds a $2$-simplex in $\lk_X(v)$. 
Now, $(\lk_{\partial X}(v))^1 \subseteq H$ so $\lk_{X}(v)$ and the $2$-simplices bounded by edges of $H$ are both homeomorphic to a disc $D^2$ with boundary $\lk_{\partial X}(v)$, the only difference being that $\lk_{X}(v)$ may contain subdivisions of $2$-simplices bounded by edges of $H$. This in particular gives a planar drawing of $H$.

Suppose now that $H$ is a proper subcomplex of $(\lk_{X}(v))^1$.
Since $(\lk_{\partial X}(v))^1\subseteq H$ and every triangle of $H$ bounds a $2$-simplex in $\lk_{X}(v)$, there must be a face $F$ of $H$ that is triangulated in $(\lk_{X}(v))^1$, meaning that there is a single vertex adjacent to every vertex in the triangle.
But such a vertex forms a clique of size $4$ in the graph of $\lk_{X}(v)$ which, since $X$ is $\cato$, implies that there is a $4$-dimensional cell in $X$, a contradiction.\qedhere
\end{enumerate}
\end{proof}

Recall that $G(X)$ is the $1$-distance subgraph of $X$ induced by the vertices of $\partial X$.
These facts allow us to diagnose structures appearing in $X$ from adjacencies in $G(X)$ and partial information on the structure of $X$.
We elaborate on this in the next lemma.

A \defn{cut-vertex} in $X$ is a vertex $v\in X^0$ such that $X\setminus v$ has at least two non-empty connected components.
Recall that a \defn{corner of a cube} in $X$ is a vertex $v$ contained in a unique cube of $X$. In particular, $v\in X^0$ and $\deg_X(v)=3$ and $v$ is incident to a unique cube in $X$, whose three faces incident to $v$ lie on $\partial X$.
\begin{lemma}\label{lem:basiclinks}
Let $X$ be a finite \cato\ cube complex with an embedding in $\R^3$ and $v\in (\partial X)^0$.
\begin{enumerate}[label=(\alph*), font=\textup]
\item \label{lem:basiclinks1} If a vertex $v$ is a cut-vertex of $X$, then $\lk_{\partial X}(v)$ is disconnected.
\item \label{lem:basiclinks2} If $\lk_{\partial X}(v)$ is a triangle, then $v$ is a corner of a cube. 
\item \label{lem:basiclinks3} If $\deg_{G(X)}(v)=4$ and $v$ is in a cube of $X$, then $\deg_{\partial X}(v)=4$, which implies that $\lk_{\partial X}(v)$ is either a cycle of length $4$, a triangle with a pendant edge, or a triangle plus an isolated vertex.
Moreover, one of the following holds:
\begin{itemize}
\item $v$ is incident to exactly one cube and one free face of $X$;
\item $v$ is incident to exactly one cube and one edge of $X$ not contained in any face;
\item $v$ is incident to at least four cubes of $X$, and the four faces incident to $v$ on $\partial X$ each lie in different cubes; or 
\item $v$ is incident to exactly two cubes in $X$, each of which contains two of the four faces incident to $v$ on $\partial X$.
\end{itemize}
\end{enumerate}
\end{lemma}
\begin{proof}
\begin{enumerate}[label=(\alph*), font=\textup]
\item If $\lk_{\partial X}(v)$ is connected, then by Lemma~\ref{lem:basic} \ref{lem:basic1} so is $\lk_{X}(v)$. Hence $X\setminus v$ is also connected.
\item By Lemma~\ref{lem:basic} \ref{lem:basic4} with $H=\lk_{\partial X}(v)$, $\lk_{\partial X}(v)\simeq (\lk_X(v))^1$ and so $\lk_{X}(v)$ is a single triangle as desired.
\item Suppose for contradiction that $\deg_{\partial X}(v)\neq 4$. Then, $\deg_{\partial X}(v)\leq 3$.
Since $v$ is in a cube of $X$, by Lemma~\ref{lem:basic} \ref{lem:basic2} $\lk_{\partial X}(v)$ has exactly $3$ vertices. In particular, $\lk_{\partial X}(v)$ is connected: otherwise each connected component of $\lk_{\partial X}(v)$ would have fewer than $3$ vertices and hence no cycles, this would lead to a contradiction in view of Lemma~\ref{lem:basic} \ref{lem:basic3} as $\lk_{X}(v)$ contains at least one $2$-simplex.
Hence, by Lemma~\ref{lem:basic} \ref{lem:basic4} $v$ is then the corner of a cube, so $\deg_{G(X)}=3$, a contradiction.

If $\lk_{\partial X}(v)$ does not contain a cycle then Lemma \ref{lem:basic} \ref{lem:basic3} contradicts the fact that $v$ lies in a cube. Thus $\lk_{\partial X}(v)$ is either a $4$-cycle, a $3$-cycle with a pendant edge, or a $3$-cycle plus an isolated vertex.

In the first case, if two faces incident to $v$ lie in the same cube, then the corresponding edges in $\lk_X(v)$ necessarily share an endpoint, and lie in a triangle.
Hence, $(\lk_{X}(v))^1$ contains a $4$-cycle with an extra edge. 
By Lemma~\ref{lem:basic} \ref{lem:basic4}, $\lk_X(v)$ is a $4$-cycle with an extra edge, with each triangle bounding a $2$-simplex. Otherwise the four faces of $\partial X$ containing $v$ lie in four different cubes.

In the second case, the triangle bounds a $2$-simplex and the final edge corresponds to a free face.

Similarly, in the last case the triangle bounds a $2$-simplex and the remaining vertex corresponds to an edge not contained in a face.\qedhere
\end{enumerate}
\end{proof}

We conclude this technical section with a standard observation that in \cato\ cube complexes, cycles of length $4$ bound a face.
\begin{lemma}\label{lem:squaresarefilled}
Let $w,x,y,z$ be vertices forming a square in the graph of $X$.
Then these vertices lie in a face of $X$.
\end{lemma}
\begin{proof}
Let $C$ be the square they form. 
Since $X$ is simply connected, by van Kampen's theorem (see \cite[Lemma~3.1]{wise}) there is a disc diagram $D\rightarrow X$ with $C\simeq \partial D$.
Let $D$ be such a disk, chosen to minimise its number of faces.
Under such minimality assumptions, $D$ contains in particular no nonogons, no bigons and its hyperplanes do not self-cross (see \cite[Lemma~3.2]{wise}). 
Since $C$ has only $4$ edges and any hyperplane crosses the boundary twice, $D$ admits at most two hyperplanes and hence has at most one face:
if two distinct faces share an edge, their four midcubes belong to three distinct hyperplanes of $D$, as they would otherwise force a forbidden structure in $D$.
\end{proof}
In particular, pairs of vertices with two common neighbours correspond to faces of $X$.
As a consequence, we can easily find the neighbours of a vertex $v\in \partial X$ which lie in the same face: they are precisely those that have a common neighbour in $G(X)$ other than $v$.
Hence, from the $1$-skeleton of $\partial X$ we can recover the link in $\partial X$ of every vertex in $\partial X$.

\section{Cleaning}\label{sec:cleaning}
Let $X$ be a finite \cato\ cube complex.
Recall that a vertex is the \defn{corner of a cube} if it has degree $3$ in $X$ and is contained in a unique cube in $X$. Analogously, a \defn{corner of a face} is a vertex with degree $2$ in $X$ that is contained in a unique face in $X$.
Recall that \defn{cut-vertex} in $X$ is a vertex $v\in X^0$ such that $X\setminus v$ has at least two non-empty connected components.
Each such connected component together with $v$ has fewer boundary vertices than $X$.

In this section, we describe substructures with their recognition and reduction steps in \cato\ cube complexes.
In the order that they will be performed, the cleaning operations are the following:
\begin{enumerate}[label={(\arabic*)}, noitemsep]
\item removing cut-vertices;
\item removing corners of faces;
\item removing vertices of degree $3$ that are not in a cube.
\end{enumerate}
For (1), the idea of reduction is that if $v$ is a cut-vertex then we will try to apply the induction hypothesis to each connected component of $X\setminus v$ with $v$ added back. If $v$ is one of the features in (2) and (3), then we will apply the induction hypothesis to $X-v$. In order to apply induction, it is important to note that each of the above reduces the number of boundary vertices by at least one. A \cato\ cube complex with none of the above features is called \defn{clean}.

Note that the order of our cleaning operations is important in the sense that when we go through the steps for a later structure in the list, we sometimes need the assumption that none of the earlier structures are present. Likewise, it is important that we can later assume that our complex is clean to then show that there exists a row of cubes on the boundary.

\subsection{Removing cutvertices}
\textit{Recognition}. A vertex $v\in \partial X^0$ is a cut-vertex in $X$ if and only if it is a cut-vertex in the graph of $\partial X^1$.

\begin{proposition}
Let $X \hookrightarrow \mathbb{R}^3$ be a \cato\ cube complex.
Then $X\setminus v$ is connected if and only if $\partial X \setminus v$ is connected.
\end{proposition}
We stress that the \cato\ condition is necessary: consider for instance any cubulation of the space $Y = \{(x,y,z) \colon x^2+y^2+z^2 \leq 4 \text{ and } (x-1)^2+y^2+z^2\geq 1\}$ consisting of the points between two spheres meeting at the single point $p=(2,0,0)$. Note that $Y$ is not contractible. Then $Y\setminus p$ is connected while $\partial Y\setminus p$ is disconnected.

\begin{proof}
For any cube complex $Z$ and vertex $v\in Z$, note that $Z\setminus v$ is connected if and only if $Z^1\setminus v$ is connected.
Hence, it suffices to show the assertion for $1$-skeleta: that $X^1\setminus v$ is connected if and only if $\partial X^1\setminus v$ is connected.

If $X\setminus v$ is disconnected, then clearly $\partial X\setminus v$ is disconnected.
For the converse, suppose that $\partial X\setminus v$ has at least two connected components. Let $A$ be such a component and write $B\coloneqq \partial X\setminus (v \cup A)$.
Note that since the restriction of an immersed hyperplane of $X$ to $\partial X$ is connected (it is a closed walk) and avoids $X^0$, it cannot intersect both $A$ and $B$. At the same time, it must intersect one of the two.
Consider two immersed hyperplanes $H, H'\hookrightarrow X$ such that $H$ and $H'$ restricted to the boundary of $X$ are contained in $A$ and $B$ respectively. 
By Theorem~\ref{thm:wise} the hyperplanes $H$, $H'$ are \cato~cube complexes, and further their top dimension is at most $2$.
Recall that nonogons in a disc diagram are hyperplanes whose image is a cycle, and that a self-crossing hyperplane is one which contains two midcubes of some face.

\begin{claim}
The immersed hyperplane $H$ contains no nonogons. 
\end{claim}
\begin{poc}
We first show that internal vertices of any disc diagram $D\hookrightarrow H$ must have degree at least $4$. 
Observe that a vertex $v\in H$ of degree $1$ corresponds to the edge $v\times I\in X$ incident to a single face, and therefore is on $\partial X$. A vertex $v\in H$ of degree $2$ is incident to at most one face in $H$ as this would otherwise create a bigon in the link of vertex $v\times \{0\}\in X$, and thus is on $\partial H$. Lastly, if a vertex $v\in H$ has degree $3$ and is incident to three faces $F_1,F_2,F_3$ in $H$, then the faces $F_i\times \{0\}$ of $X$ must all lie in some single cube $C\in X$.
Since the cubical neighbourhood $N$ of $H$ is convex, the vertex of $C$ not contained in any $F_i\times\{0\}$ also lies in $N$ and hence $C\subseteq N$.
In particular, this means that the immersed hyperplane $H\hookrightarrow X$ contains all three midcubes $F_i\times \{1/2\}$ of $C$, which is impossible as they pairwise share edges.

Now suppose that there is a nonogon and let $D\hookrightarrow H$ be a minimal disc diagram containing it, so that $\partial D$ is a ring of the nonogon, and all of the vertices on $\partial D$ have degree exactly $3$.
Then \cite[Lemma 3.4]{john2D} implies that any disc quadrangulation with all internal degrees at least 4 has a boundary vertex of degree 2, a contradiction.
\end{poc}
\begin{claim}\label{claim:hyperplanesdonotcross}
$H\cap H'=\emptyset$.
\end{claim}
\begin{poc}
Suppose not. Then, since the cubical neighbourhood of $H\cap H'$ in $H$ is a sequence of faces where consecutive ones share edges, there is a disc diagram $D\hookrightarrow H$ which contains $H\cap H'$. 
As $H$ has no nonogons and is finite, $H\cap H'$ must intersect $\partial H$ non-trivially, contradicting that $H'\subseteq B$.
\end{poc}
Let $a\in A$ and $b\in B$ be neighbours of $v$ and suppose for contradiction that there is an $(a,b)$-path $\gamma \subseteq X^1\setminus v$.
As $X$ is contractible, the closed loop formed by $\gamma$ together with $v$ bounds a disc diagram $D\hookrightarrow X$.
Let $F_1,\dotsc, F_k$ and $av=e_0,\dotsc,e_k=bv$ be the clockwise ordering of faces and edges, respectively, incident to $v$ in $D$.
For each $0\leq i\leq k$ write $H_i$ for the hyperplane corresponding to the midcube of edge $e_i$.
Then, $H_{i-1}$ crosses $H_{i}$ in face $F_i$ for each $i=1,\dotsc, k$.
Since the restrictions to the boundary of $H_0$ and $H_k$ lie in $A$ and $B$ respectively, this contradicts Claim~\ref{claim:hyperplanesdonotcross} (with $H\coloneqq H_0$ and $H'\coloneqq H_k$).
\end{proof}

\textit{Reduction}.
For this step, we use a simple fact from algebraic topology which we prove for completeness. An alternative argument using van Kampen's theorem is also possible.
\begin{lemma}\label{lem:wedge}
Let $X= \bigwedge_i X_i$ be a wedge of CW complexes $X_i$ with common point $x_0$.
Then if $X$ is contractible, so is each $X_i$.
\end{lemma}
\begin{proof}
The map $\pi_n(\bigwedge_i X_i)\rightarrow \pi_n(\prod_i X_i)\simeq \bigoplus_i \pi_n (X_i)$ induced by inclusion is surjective for each $n$.
By assumption the term on the left hand side is trivial and thus each $\pi_n(X_i)$ is trivial as well for each $n$.
\end{proof}

Suppose that $v$ is a cut-vertex. Then $X$ can be written as the wedge $\bigwedge_i X_i$ of finitely many subcomplexes $X_i$ whose pairwise intersection is $\{v\}$.
By Lemma~\ref{lem:wedge}, each $X_i$ is contractible.
Since each connected component of $\lk_X(v)$ is a flag complex, so is each $\lk_{X_i}(v)$ and thus each $X_i$ is \cato.

It is clear that $\partial X_i\subseteq \partial X$ and $X$ can be reconstructed from its subcomplexes $X_i$.
Additionally, if vertices $x,y\in \partial X$ lie in the same subcomplex $X_i$, then $d_{X_i}(x,y) = d_{X}(x,y)$ as any $x,y$-path in $X$ using a vertex not in $X_i$ can be shortened to a path using only vertices in $X_i$.

\subsection{Removing corners of faces}
\textit{Recognition}. If $X$ has no cut-vertices, a vertex $v\in \partial X^0$ is the corner of a face precisely if $d_{\partial X}(v) = 2$.
Indeed, if $v$ is a corner of a face then $d_{\partial X}(v)=2$ by definition.
Conversely, if $d_{\partial X}(v) = 2$ then by Lemma~\ref{lem:basic} \ref{lem:basic2}, $v$ is not contained in any cube of $X$ and so by Lemma~\ref{lem:basic} \ref{lem:basic3}, $\lk_{\partial X}(v)\simeq \lk_X(v)$ and therefore is a single edge.

\textit{Reduction}.
Suppose that $v$ is a corner of a face $C$. Note that $C$ must be free, and hence $C$ is in $\partial X$.
Remove all cells containing $v$ to obtain a proper subcomplex $Y\subseteq X$.
As $X$ deformation retracts onto $Y$ and $X$ is contractible, $Y$ is contractible. 
For a vertex $u$ not in $C$, we have $\lk_Y(u)= \lk_X(u)$ which is therefore a flag complex. 
For a vertex $w$ in $C$, $\lk_Y(w)$ is obtained from $\lk_X(w)$ by removing either a degree 1 vertex together with its incident edge, or an edge not contained in any triangle. 
In either case, the resulting complex $\lk_Y(w)$ is flag and so $Y$ is \cato.
It is clear that $X$ can be reconstructed from $Y$. We also have that $\partial Y\subseteq \partial X$ because all vertices of $C$ are in $\partial X$. Moreover, $d_{Y}(x,y)=d_{X}(x,y)$ for any $x,y\in \partial Y^0$. To see this, let $u$ be the vertex of $C$ not adjacent to $v$ and note that $v$ can be replaced by $u$ in any shortest $(x,y)$-path in $X$ without changing its length. 

\subsection{Removing vertices of degree 3 not in a cube}
\textit{Recognition}.
Suppose $X$ has no cut-vertices nor vertices of degree $2$, and $v\in\partial X^0$ has degree $3$.
Then $v$ is not in a cube of $X$ if and only if two of its neighbours do not have a common neighbour different from $v$.
Indeed, since $X$ has no cut-vertices, it follows from Lemma~\ref{lem:basiclinks} \ref{lem:basiclinks1} that $\lk_{\partial X}(v)$ is connected and is therefore either a path with $2$ edges or a triangle.
The former case occurs precisely when $v$ is the unique common neighbour of two of its neighbours, and in the latter case, by Lemma~\ref{lem:basiclinks} \ref{lem:basiclinks2}, $v$ is in a cube.

\begin{figure}[ht]
    \centering
    \includegraphics[scale=0.8]{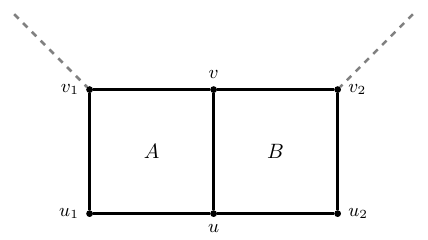}
    \caption{Pattern corresponding to a vertex $v$ of degree $3$ not in a cube.}
    \label{fig:2squares}
\end{figure}

\textit{Reduction}.
Suppose $X$ has no cut-vertices nor vertices of degree $2$, and that $v\in \partial X^0$ is a vertex of degree $3$ not in a cube. This forces the cells containing $v$ to form them pattern depicted in Figure~\ref{fig:2squares}, where all vertices in the figure are on $\partial X$.
Let $Y\subseteq X$ be the proper subcomplex of $X$ obtained by removing all cells containing $v$. Using the known pattern that, we can reconstruct $X$ from $Y$.

Since $X$ deformation retracts onto $Y$ via the map collapsing $v$ onto $u$ and the edges $v_1v$, $vv_2$ onto $v_1u_1u$, $v_2u_2u$ respectively, $Y$ is contractible.
Moreover, since $Y$ is obtained from $X$ by removing cells of dimension at most $2$ not contained in any cubes, all links of vertices in $Y$ are flag complexes. 
Hence, $Y$ is a \cato\ cube complex.

Consider vertices $p,q\in \partial Y$; we certainly have $d_{Y}(p,q) \geq d_{X}(p,q)$.
Let $H$ be the unique immersed hyperplane splitting $X$ into parts $X_u$, $X_v$ with $u\in X_u$, $v\in X_v$. Note that we can recognise the part to which each vertex on $\boundary X$ belongs from the known distance matrix.
By convexity of the neighbourhood of $H$, $p\in X_u$ if and only if $d_X(p,u)<d_X(p,v)$. There are three cases to consider to recover distances exactly.
\begin{enumerate}
\item If $p,q\in X_u$, then since hyperplanes are convex (Theorem~\ref{thm:wise} \ref{thm:wise3}), no shortest $p,q$-path uses $v$.
Hence, $d_{Y}(p,q) = d_{X}(p,q)$.

\item If $p\in X_u, q\in X_v$, we can assume $d_X(p,q)=d_X(p,v)+d_X(v,q)$ as every shortest $p,q$-path otherwise misses $v$.

Since $v,q\in X_v$, either $d_X(v,q)=1+d_X(v_1,q)$ or $d_X(v,q)=1+d_X(v_2,q)$; assume without loss of generality the latter. As $H$ is the only hyperplane separating $u_2$ and $v_2$, we have $d_X(p,u_2)=d_X(p,v_2)-1=d_X(p,v)$ and $d_X(u_2,q)=d_X(v_2,q)+1=d_X(v,q)$. Thus, there is a shortest $(p,q)$-path containing $u_2$. This path does not contain $v$, since a shortest $(p,u_2)$-path lies entirely within $X_u$ and $v$ is too far from $q$ to be on a shortest $(u_2,q)$-path.

\item Finally, suppose that $p, q\in X_v$. To proceed, we claim that every shortest $(p,q)$-path in $X_v$ is a shortest path in $X$. This holds since if $\rho$ is a $(p,q)$-path that intersects $X_u$ in a subpath $\rho'_u$, the convexity of the tubular neighbourhood of $H$ means that $\rho'_u$ is contained in this neighbourhood. Then we can replace $\rho'_u$ by a projection of the same length $\rho'_v$ in $\rho$, and this produces a $(p,q)$-path contained in $X_v$ that is shorter than $\rho$. 

Now note that $v$ is a cut-vertex in $X_v$. If $p$ and $q$ are in the same component of $X_v\backslash v$ (this can again be recognised from the distance matrix for $\boundary X$), then the shortest $p,q$ path in $X_v$ avoids $v$. The preceding claim then implies that $d_Y(p,q) = d_X(p,q) = d_X(p,q)$. 

So suppose that $p$ and $q$ are in different components of $X_v\backslash v$, meaning a shortest $(p,q)$-path in $X$ uses $v$. By Theorem~\ref{thm:wise} $H$ is simply connected, so by van Kampen's theorem (see \cite[Theorem~1.20]{hatcher}), $X_v$ is also simply connected. Furthermore, since $v$ is not contained in any face in $X_v$, any $(v_1,v_2)$-path $P$ avoiding $v$ would form a non-trivial loop with the path $v_1vv_2$, which is impossible.
Hence, any $(p,q)$-path $\gamma$ in $X^1$ avoiding $v$ -- i.e.\ a path in $Y^1$ -- must use at least one vertex outside $X_v$. 

Let $xy$ be an edge of $\gamma$ with $x\in X_v$, $y\in X_u$. 
Then, again by the convexity of the tubular neighbourhood of $H$, $d_X(p,y) \geq d_X(p,x)+1$ and $d_X(y,q) \geq d_X(x,q)+1$. 
Hence, $|\gamma| \geq d_X(p,y) + d_X(y,q) \geq d_X(p,x) + d_X(x,q) +2$.
In particular, taking $\gamma$ to be a shortest path allows us to conclude that $d_{Y}(p,q) \geq d_{X}(p,q) +2$.
In fact, by replacing $v$ by $u_1uu_2$ in $\gamma$, we see that $d_{Y}(p,q) = d_X(p,q)+2$.
\end{enumerate}

\section{Rows of cubes}\label{sec:rowsofcubes}
A \defn{row of cubes of length $k$} in $X$ is a tuple of cubes $(C_1,\dotsc, C_k)$ from $X$ where non-consecutive cubes are disjoint, and for each $i=2,\dotsc, k-1$, there are opposite faces $F$, $F'$ of $C_i$ such that $C_{i-1}\cap C_i = F$ and $C_i\cap C_{i+1}=F'$.
If moreover, there is a path $p_0\dotsb p_k$ such that $\deg_X(p_0)=\deg_X(p_k)=3$ and $\deg_X(p_i) = 4$ for $i=1,\dotsc,k-1$, and additionally $p_0\in C_1$, $p_k\in C_k$ and $p_i\in C_{i-1}\cap C_i$ for each $i=2,\dotsc,k$, then we say that $(C_1,\dotsc, C_k)$ is \defn{on the boundary}. In particular, this implies that $p_0, \ldots, p_k$ are boundary vertices.

In this section, we describe the recognition and reduction steps for rows of cubes on the boundary.
For the former, we will introduce so-called `row configurations' and `good row configurations' in Section~\ref{sec:scs}.
As will follow from Lemma~\ref{lem:scrowcorrespondance}, when $X$ is a contractible clean \cato\ cube complex, good row configurations in $G(X)$ correspond to rows of cubes on the boundary.
Their existence in clean contractible \cato\ cube complexes will be discussed in Section~\ref{sec:thethickening}.
The latter step is detailed in Section~\ref{sec:removingrowofcubes}.

\subsection{Row configurations}\label{sec:scs}
\begin{definition}
A \defn{row configuration of length $k$} in a graph $H=(V,E)$ is a tuple of labelled vertices 
\[
( a_i, p_i, b_i \colon i=0,\dotsc, k) 
\]
for some $k\geq 1$ such that
$\deg_{H} (p_0)=3$,
$\deg_{H} (p_k)\neq 4$,
$\deg_{H} (p_i) = 4$ for $i=1,\dotsc,k-1$,
all edges (referred to as the \defn{edges of $\mathcal{C}$})
\begin{align*}
a_ip_i, b_ip_i, \qquad &i = 0,\dotsc, k;\\
a_{i-1}a_i, b_{i-1}b_i, p_{i-1}p_i, \qquad &i=1,\dotsc,k;
\end{align*}
are in $E$, the $p_i$ are distinct and $(a_i)$, $(b_i)$ are distinct sequences. As it turns out, these sequences do not intersect in row configurations on the boundary of \cato\ cube complexes.
The \emph{oriented} path $p_0p_1\dotsc p_k$ is called the \defn{spine} of $\mathcal{C}$ and we refer to $p_0$ and $p_k$ as the \defn{start vertex} (which always has degree 3) and \defn{end vertex} (which always has degree at least 3), respectively. 
We identify row configurations with the same spine, i.e.\ those of the form
\begin{align*}
& (a_i, p_i, b_i \colon i=0,\dotsc, k) \\
& (b_i, p_i, a_i \colon i=0,\dotsc, k).
\end{align*}
We refer to row configurations with end vertex of degree $3$ as \defn{good row configurations} or good configurations for short.
When $X$ is a cube complex, we will say `a (good) row configuration in $X$' to mean a (good) row configuration in the graph $\partial X^1$.
\end{definition}

\begin{figure}[ht]\label{fig:sc}
	\centering
        \includegraphics[scale=0.8]{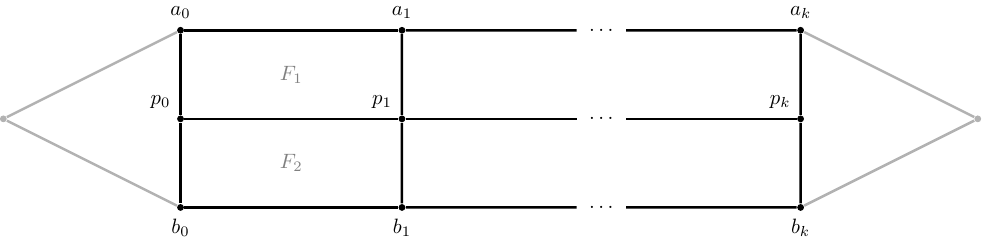}
	\caption{Row configuration.}
\end{figure}

When $Y$ is a clean \cato\ cube complex, row configurations in $G(Y)$ enjoy multiple useful properties: we show that all edges of a row configuration in $G(Y)$ are in fact edges of $(\partial Y)^1$ and that a row configuration corresponds to a row of cubes. 

\begin{lemma}\label{lem:scrowcorrespondance}
Let $Y$ be a clean \cato\ cube complex. 
Then a row configuration of length $k$ in $(\partial Y)^1$ or $G(Y)$ corresponds to a row of cubes of length $k$ on the boundary of $Y$.
Namely, given a row configuration
$\mathcal{C}\coloneqq (a_i, p_i, b_i \colon i=0,\dotsc, k)$ in $(\partial Y)^1$ or $G(Y)$,
there is a row of cubes $\mathcal{R}\coloneqq(C_1,\dotsc, C_k)$ in $Y$ such that the faces bounded by $a_{i-1},a_i,p_i, p_{i-1}$ and $p_{i-1},p_i, b_{i},b_{i-1}$ are faces of $C_i$ for each $i=1,\dotsc, k$, and all edges of $\mathcal{C}$ are in $(\partial Y)^1$.
Further, this row of cubes is uniquely determined by $k$ and the vertices $p_0$, $p_1$.
We refer to $\mathcal{R}$ as the \defn{underlying row of cubes} of $\mathcal{C}$.
\end{lemma}
Note that a single row of cubes may be the underlying row of cubes of multiple row configurations in $Y$ (e.g. when $Y$ is a single cube).
However, such row configurations are uniquely determined by the first edge of their spine.
\begin{proof}
Suppose that $\mathcal{C}$ is a row configuration in $G(Y)$.
Since $Y$ is clean $p_0$ is in a cube, so by Lemmas~\ref{lem:basic} \ref{lem:basic2} and \ref{lem:basiclinks} \ref{lem:basiclinks2}, $p_0$ is the corner of a cube $C_1$ in $Y$ and moreover edges incident to it in $G(Y)$ are present in $(\partial Y)^1$.
In particular, $C_1$ has faces $F_1$, $F_2$ containing vertices $\{a_0, a_1, p_1, p_0\}$ and $\{b_0, b_1, p_1, p_0\}$, respectively. 
If $a_0=b_0$, then $F_1$ and $F_2$ share three vertices and so, as $Y$ is \cato, $F_1=F_2$ and in particular $a_1=b_1$. Continuing in this way we see that $(a_i)=(b_i)$, a contradiction. In a similar way, if there is some $0<i\leq k$ with $a_i=b_i$ we can see that $a_0=b_0$, leading to a contradiction.
So without loss of generality, $a_0\neq b_0$.
As $p_0$ has degree $3$ in $Y$, both $F_1$ and $F_2$ are in $\partial Y$ and so $a_0a_1$, $b_0b_1$ are edges of $(\partial Y)^1$.

If $\deg_{G( Y)}(p_1)\neq 4$, we are done.
Otherwise, $\deg_{G( Y)}(p_1)=4$ and faces $F_1$, $F_2$ both contain $p_1$ and are faces of the same cube $C_1$. Furthermore, $\partial Y^1$ contains the $4$-cycles $p_1a_1a_2p_2$ and $p_1b_1b_2p_2$, which by Lemma~\ref{lem:squaresarefilled} must each span faces of $Y$, and so $\lk_{\partial Y}(p_1)$ contains a $4$-cycle. The only possibility from Lemma~\ref{lem:basiclinks} \ref{lem:basiclinks3} consistent with this is that there is some cube $C_2$ such that $p_1$ is only incident to $C_1$, $C_2$ in $Y$, i.e.\ $C_1$, $C_2$ share a face containing $p_1$.
Moreover, both faces of $C_2$ incident to $p_1$ are on $\partial Y$, and in particular, the edges of these faces are also on $\partial Y$.
Let $\{a_1, a_2, p_2, p_1\}$ and $\{b_1, b_2, p_2, p_1\}$, respectively, be the vertices of the faces of $C_2$ incident to $p_1$.
Continuing in this way we find the desired row of cubes $(C_1,\dotsc, C_k)$ with all edges of $\mathcal{C}$ on $\partial Y$.
\end{proof}
In light of this lemma, there is a one to one correspondence between row configurations in $G(Y)$ and $(\partial Y)^1$: a row configuration in $G(Y)$ is clearly one in $(\partial Y)^1$ since all its edges are in $(\partial Y)^1$, and conversely, given a row configuration $\mathcal{C}$ in $(\partial Y)^1$ labelled as above, we must have $\deg_{\partial Y}(p_i)=\deg_{Y}(p_i)$ for each $i$ and hence $\mathcal{C}$ is a row configuration in $G(Y)$ as well.
Moreover, the underlying row of cubes of a good configuration in $(\partial Y)^1$ is on the boundary, and rows of cubes on the boundary give rise to good configurations in $(\partial Y)^1$.

\subsection{Removing a row of cubes}\label{sec:removingrowofcubes}
Let $\mathcal{C} = (a_i,p_i,b_i\colon i=0,\dotsc,k)$ be a good configuration in $\partial X^1$ and $(C_1,\dotsc, C_k)$ its underlying row of cubes as in Lemma~\ref{lem:scrowcorrespondance}.
We define $Y\coloneqq X- \mathcal{C}$ to be the cube complex obtained from $X$ by removing vertices $p_0,\dotsc, p_k$ as well as all cubes, faces and edges containing them.
For each $i$, we denote the common neighbour of $a_i$ and $b_i$ by $c_i$.
Note that $\partial Y$ has at most $k-1$ new vertices $c_1,\dotsc, c_{k-1}$ and thus has at least one fewer vertex than $\partial X$.

We now check that $Y$ admits an embedding in $\mathbb{R}^3$, is contractible and all links are flag.
The first part is clear: an embedding $X \hookrightarrow \mathbb{R}^3$ induces an embedding $Y \hookrightarrow\mathbb{R}^3$.
The second part is also straightforward as $X$ deformation retracts onto $Y$ and thus has the same homotopy type.
For the third part, observe that for each $i= 1,\dotsc, k-1$, $\lk_{Y} (c_i)$ is obtained from $\lk_{X}(c_i)$ by removing an edge and the only two triangles containing it.
When $i=0$ or $k$ then $\lk_Y(c_i)$ is obtained from $\lk_X(c_i)$ by removing an edge as well as the unique triangle containing it.
In particular, in either case we have that $\lk_Y(c_i)$ is flag, as desired.
A similar, simpler argument shows that all $\lk_Y(a_i)$s and $\lk_Y(b_i)$s are also flag.

It remains to recover the boundary distances.
To do so, we will use the following fact, which essentially says that the distance between a vertex $p$ and a pair of adjacent vertices separated by a hyperplane $H$ determines on which side of $H$ the vertex $p$ lies.
In fact, this lemma holds for general \cato\ cube complexes, with essentially the same proof.
\begin{figure}[ht]\label{fig:hyperplaneedge}
	\centering
        \includegraphics[scale=0.8]{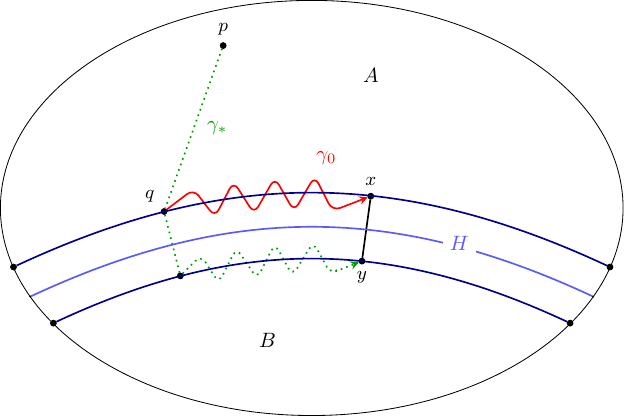}
	\caption{Sketch of proof of Lemma~\ref{lem:adjacentverticesandhyperplanes}. If the dashed path $\gamma_*$ is a geodesic, we can project its intersection with the tubular neighbourhood of $H$ to the side of $A$, giving a subpath $\gamma_0$. This produces a $(p,x)$-path of length at most $d_X(p,y)-1$.}
\end{figure}
\begin{lemma}\label{lem:adjacentverticesandhyperplanes}
Let $x,y$ be adjacent vertices of $X$ and $H\hookrightarrow X$ the immersed hyperplane containing the midcube of the edge $xy$.
Let $A$, $B$ be the connected components of $X\setminus H $, containing $x$ and $y$ respectively.
Then, for every $p\in X$,
\begin{equation*}
d_X(p,y) = d_X(p,x) +1 \Leftrightarrow p\in A.
\end{equation*}
\end{lemma}

\begin{proof}
Suppose first that $p\in A$.

Given a geodesic path from $p$ to $x$ in $X^1$, we can extend it using the edge $xy$ to form a $(p,y)$-path of length $d_X(p,x) +1$, and hence $d_X(p,y) \leq d_X(p,x) +1$.

For the reverse inequality, consider a geodesic path $\gamma_*$ from $p$ to $y$ in $X^1$.
Let $q$ be the first intersection point of $\gamma_*$ with the tubular neighbourhood $N \coloneqq H\times [0,1]$ of $H$, where $H\times \{0\}\subseteq A$, $H\times \{1\} \subseteq B$.
In particular, if $(1,h) = y$ for $h\in H$ then $(0,h)=x$.
The subpath $\gamma$ of $\gamma_*$ from $q$ to $y$ has both endpoints in the tubular neighbourhood of $H$ and so, by Theorem~\ref{thm:wise}, lies entirely in it.

We show that the $(p,y)$-path obtained from $\gamma_*$ by `pushing' $\gamma$ to $H \times \{0\}$ uses at least one less edge than $\gamma_*$.
More precisely, let $\gamma_0\coloneqq \{ (0,h) \colon (t,h)\in N\}$ be the projection of $\gamma$ to $H \times \{0\}$ and define $\gamma'$ to be $\gamma_*$ from $x$ to $q$ concatenated with $\gamma'$ and $xy$.
The length of $\gamma_0$ is at most the number of edges of $\gamma$ which do not cross $H$, so $|\gamma_0|<|\gamma|$ as $\gamma$ has at least one edge crossing $H$.
In particular, $|\gamma'|< |\gamma|$ and therefore $d_X(p,x) \leq d_X(p,y)-1$, as desired.

Applying the symmetric argument when $p\in B$ yields the desired equivalence.
\end{proof}

We use the labelling described in the beginning of this section. 
Additionally, write $N$ for the hyperplane of $X$ containing the midcube of edge $b_ip_i$.

\begin{lemma}\label{lem:recoverdistances}
Given the matrix $D_{\partial X}$ of distances in $X$ between vertices of $\partial X$, we can deduce the matrix $D_{\partial Y}$ of distances in $Y$ between vertices of $\partial Y$.
\end{lemma}
\begin{proof}
First, we claim that for any vertices $x,y \in \partial Y$, we can find a geodesic path avoiding the removed vertices $\{p_i\colon i=1,\dotsc,k\}$.
Indeed, given $\gamma$ a geodesic $(x,y)$-path in $X$, by replacing each $p_i$ in $\gamma$ by $c_i$ we obtain a walk in $Y$ of the same length.
In particular, for any $x,y \in \partial Y$ we immediately have $d_X(x,y) = d_Y(x,y)$.
These distances are already given by the distance matrix $D_{\partial X}$ for vertices $x,y\in \partial Y\cap \partial X$.

It remains to deduce the distances between boundary vertices and those newly created by the removal process.
For each $i = 1,\dotsc, k-1$, we claim that for any $x\in \partial Y$, the distance to vertex $c_i$ is given by
\[
d_{Y}(c_i, x) = \begin{cases}
d_X(x, a_i)+1 &\text{ if } d_X(x, b_i)  = d_X(x, p_i)+1,\\
d_X(x, a_i)-1 &\text{ if } d_X(x, b_i) = d_X(x, p_i)-1.
\end{cases}
\]
Let $A,B$ be the two connected components of $X\setminus N$.
We apply Lemma~\ref{lem:adjacentverticesandhyperplanes} in $X$ to edges $p_ib_i$ and $a_ic_i$ (whose midcubes are both contained in hyperplane $N$).
On one hand this allows us to determine the component of $x$ as $d_X(x,b_i)$, $d_X(x,p_i)$ are known and
\begin{align*}
d_X(x,b_i) &= d_X(x,p_i) +1 \Leftrightarrow x\in A,\\
d_X(x,b_i) &= d_X(x,p_i) -1 \Leftrightarrow x\in B.
\end{align*}
On the other hand, this information allows us to deduce $d_X(x,c_i)=d_Y(x,c_i)$ as $d_X(x,a_i)$ is known and
\begin{align*}
d_X(x,c_i) &= d_X(x,a_i) +1 \Leftrightarrow x\in A,\\
d_X(x,c_i) &= d_X(x,a_i) -1 \Leftrightarrow x\in B.\qedhere
\end{align*}
\end{proof}

\section{Rows of Cubes II: The Thickening}
\label{sec:thethickening}

The purpose of this section is to overcome the difficulty that our complex $X$ is not necessarily homeomorphic to a ball. We describe a method by which to `thicken' $X$ into a new cube complex, denoted $\thick{X}$, that is homeomorphic to the ball, but whose boundary is sufficiently similar to the original that we can use it to deduce the existence of structures on $X$. Broadly, our thickening process entails gluing a new cube on every face of our cube complex, and then identifying sides of those new cubes to reflect face incidences in $X$. This creates a `shell' of cubes around $X$ which bulks up lower-dimensional free cells and produces a pure contractible cube complex $\thick{X}$, the \defn{thickening of $\thick{X}$}.

It is good to note that although the construction of the thickening of a cube complex $X$ depends on the choice of embedding $X\hookrightarrow \R^3$, this is not an issue when proving an existence statement for a structure in $X$.
 As such, we slightly abuse notation and speak of \emph{the} thickening of $X$ to mean a \emph{fixed instance} of a thickening of $X$.

Let $X$ be a clean \cato\ cube complex and $I=[0,1]$ denote the unit interval. 
Let the \defn{sides} of $X$ refer to faces on $\partial X$ counted with multiplicity, so that free faces are counted twice. 
For each face $F\in \partial X$, let $\sides(F)$ be the collection of its associated sides, namely a multiset $\{F, F\}$ if $F$ is free and the singleton set $\{F\}$ otherwise.
 Write $\mathcal{F} \coloneqq \bigcup_F \sides(F)$ for the multiset of sides over all faces of $\partial X$.

In the first step of our thickening process, we associate to each side $S\in \mathcal{F}$ a new cube $C_S= I\times S$ with the natural gluing map $\phi_S\colon \{0\}\times S\rightarrow S\subseteq X$.
We refer to the face $\{1\}\times S\subseteq C_S$ as the \defn{external face} of $C_S$. 
As its name suggests, the external face of $C_S$ will be on the boundary of $\thick{X}$ (which we verify more formally in \cref{claim:thickbasic}).

We define an intermediate complex $X'= X\sqcup \bigsqcup C_S \Big/\bigsqcup \phi_S$, where both disjoint unions are taken over $S\in \mathcal{F}$, namely $X'$ is the complex obtained from $X$ by gluing cubes $C_S$ to each side $S$ of $\partial X$ along maps $\phi_S$.
We write $H_{t,S}\colon S\rightarrow C_S\subseteq X'$ for the map $(x,y)\mapsto (t,x,y)$, so that $H_{0,S}(S) = S\subseteq X$ and $H_{1,S}(S)$ is the external face of $C_S$ embedded in $X'$.

Note that that there exists an embedding of $X'$ in $\R^3$, with $X\hookrightarrow \R^3$ as a subcomplex.
Fix such an embedding; for each edge $e$ of $\partial X$, it induces a cyclic ordering $S_1,\dotsc, S_\ell$ of the sides of boundary faces incident to $e$.
For each $i$, there are two faces of $C_{S_i}$ which are incident to $e$ in $X'$.
One is $\{0\}\times S_i$ and we denote the other $S_i^e$.
See Figure~\ref{fig:thickeningaroundanedge}.
For each $i$, we define gluing maps $\phi_{e,S_i}$ between $S_{i}^{e}$ and $S^{e}_{i+1}$ whenever the points between $S_{i}^{e}$ and $S_{i+1}^{e}$ in the cyclic ordering around the edge $e$ are not in $X$, where indices are considered modulo $k$.

\begin{figure}[t]
    \centering
    \includegraphics[scale=0.8]{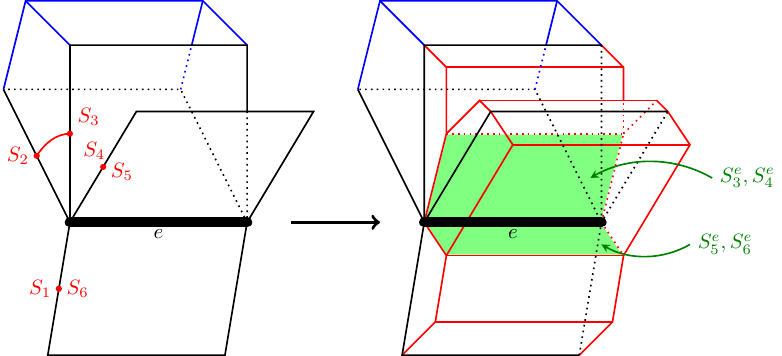}
    \caption{Example of gluing faces of cubes $C_F$ together.}
    \label{fig:thickeningaroundanedge}
\end{figure}
The \defn{thickening} of $X$, denoted $\thick{X}$, is the cube complex $\thick{X} = X'\big/ \bigsqcup \phi_{e,S}$ obtained from $X'$ by gluing cubes associated to consecutive sides in the cyclic ordering around each edge $e\in \partial X^1$ according to gluing maps $\phi_{e,S}$.
For each fixed $t\in [0,1]$ we have the following compositions:
\[
S \xrightarrow{H_{t,S}} C_S \xrightarrow{\phi_S} X' \xrightarrow{\sqcup\phi_{e,S}} \thickening,
\] 
where the union in the last map is taken over each edge $e$ of $S$.
The embedding of $X'$ in $\mathbb{R}^3$ allows us for each fixed $t\in [0,1]$ to extend these compositions to a continuous map $H_{t}\colon L\coloneqq \bigsqcup S\Big/\bigsqcup \phi_S \sqcup \bigsqcup \phi_{e,S} \rightarrow \thickening$. 
Moreover, these maps are continuous with $t$.
Then, $H_0(L) = \partial X$ and $H_1(L) \subseteq \partial \thickening$.
In fact, in Claim~\ref{claim:thickpure} we will show that $H_1(L) = \partial \thickening$, so that $\{H_t\colon t\in I\}$ describes a deformation retraction of $\thickening$ onto $X$.
We define the \defn{projection map} $\pi \colon \partial \thickening \rightarrow \partial X$ as $H_1(x) \mapsto H_0(x)$ for each $x\in L$.
Note that $\pi$ is combinatorial, and that it is well-defined as $H_1\colon L \rightarrow \thick{X}$ is an embedding.

Note that each $x\in \partial \thick{X}^0$ is adjacent to precisely one vertex $\pi (x)$ of $X^0$ in $\thick{X}$. 
We sometimes refer to the cubes $\bigsqcup_{S\in\mathcal{F}} C_S$ as the \defn{new cubes} of $\thick{X}$. 

The following claim is essentially immediate from our construction.
\begin{claim}\label{claim:thickpure}
$\thick{X}$ is a connected contractible pure 3D cube complex.
\end{claim}
\begin{poc}
The maps $\{H_t \colon t\in I\}$ give a deformation retraction of $\thick{X}$ onto $X$, so contractibility of $\thick{X}$ follows from the contractibility of $X$. Connectedness is also inherited directly from the original complex. The assumptions that $X$ is clean and simply connected mean that the only lower-dimensional cells in $X$ that are not in cubes are the free faces. Since the thickening process ensures that each of these is now in a new cube (two, in fact), we have that $\thick{X}$ is pure.
\end{poc}

\begin{claim}\label{claim:thickbasic}
The thickened complex $\thick{X}$ is a 3-manifold with boundary. Moreover, the boundary of $\thick{X}$ is precisely the union of the external faces of all new cubes glued according to restrictions of the maps $\phi_{e,S}$.
\end{claim}

\begin{poc}
Let us begin by characterising $\partial \thick{X}$. We first note that all external faces must be on the boundary of $\thick{X}$, as they are on the boundary of every intermediate cube complex as we go through the construction described. For the other direction, we will rule out the possibility that any faces, edges or vertices not in an external face can be on the boundary. To this end, note that any (open) face $F$ of $X$ is either in $\interior(X)$, or it is in $\partial X$ meaning a new cube is glued onto each side in $\sides(F)$ in the thickening process. In both cases, we see that the interior of $F$ is in the interior of $\thick{X}$ from which we conclude that the faces of $\partial \thick{X}$ are the external faces. At the same time, for any edge or vertex to be in $\partial \thick{X}$, it must be part of a face in $\partial \thick{X}$. This follows from the fact that $\thick{X}$ and hence its boundary are pure cube complexes (the latter in two dimensions). As we have just seen that all such faces are external, the claim follows.

We now show that $\thick{X}$ is a manifold. It is certainly second-countable and Hausdorff, so we just need to verify that it is locally Euclidean. Since $X$ is embedded in Euclidean space, any point in the interior of $X$ certainly has a neighbourhood homeomorphic to the open 3-ball. That leaves us to consider points on the boundary of $\thick{X}$. For points $p$ in the interior of a boundary face of $\thick{X}$ and $\eps >0$ sufficiently small, the intersection $B_\eps(p) \cap \thick{X}$ is homeomorphic to the half-ball $\{x_1,x_2,x_3: x_1^2+x_2^2+x_3^2<1, x_1\geq 0\}$ where points with $x_1=0$ map to a disc on the face (this rephrases the fact that an individual cube is homeomorphic to the ball $B^3$ and hence a 3-manifold with boundary). This is also true for points on the interior of an edge on $\partial \thick{X}$, where points with $x_1=0$ map to a disc that intersects the two (distinct, by construction of $\thick{X}$) boundary faces incident to our edge. By our characterisation of $\partial\thick{X}$ in the preceding paragraph, we know that each vertex $v$ on $\partial\thick{X}$ is surrounded by external faces of new cubes, and the link of $v$ is a disc by construction. This tells us that the cone on the link, and hence $B_\eps(v) \cap \thick{X}$, is homeomorphic to the half-ball whose boundary is the union of a disc on $\partial \thick{X}$ and $\lk(v)$. 
\end{poc}

It may be interesting to note that the boundary points of $\thick{X}$ as a manifold are precisely the points of $\partial\thick{X}$.
\begin{corollary}\label{cor:thickenball}
$\partial\thick X$ is homeomorphic to $S^2$.
\end{corollary}
\begin{proof}
Since $\thick X$ is a 3-manifold with boundary, the boundary $\partial\thick X$ is a surface. Moreover, the contractibility of $\thick X$ means that $\partial\thick X$ is a homology 2-sphere. 
The statement then follows from the classification of surfaces.
\end{proof}

\section{Existence of rows of cubes on the boundary}\label{sec:rowsonboundary}
In this section, we show the existence of good row configurations in clean \cato\ cube complexes.
In particular, Lemma~\ref{lem:scrowcorrespondance} then implies existence of rows of cubes on the boundary.

For this section, let $Y\hookrightarrow \R^3$ be a finite clean \cato\ cube complex and $\thick{Y}$ its thickening.

\subsection{Row configurations in the thickened complex}

Let $v\in \partial\thick{Y}^0$, $x\coloneqq\pi(v)$ and $C_1,\dotsc,C_k$ be the cyclically ordered cubes incident to the edge $vx$ in $\thick{Y}$.
Since $\pi$ is combinatorial, its restriction to $D$ induces a combinatorial map $\pi^*\colon \lk_{\partial \thick{Y}}(v)\rightarrow \lk_{\partial Y}(x)$.

Since $\pi^*$ is combinatorial and only identifies edges corresponding to the two sides of a single face in $Y$, we have the following fact.
\begin{fact}\label{fact:thicklink}
Under above setup, the image of $\pi^*$ is a closed walk $W \subseteq \lk_{\partial Y}(x)$. This walk is formed by identifying either pairs of edges in $\lk_{\partial\thick{Y}}(v)$ or vertices of $\lk_{\partial\thick{Y}}(v)$ together.
\end{fact}

The next lemma essentially states that the projection map $\pi$ is well-behaved at vertices of degrees $3$ and $4$: the former come from vertices of degree $3$ on $\partial Y$ and under suitable assumptions the latter come from vertices of degree $4$.
Moreover, all other vertices of $\partial \thick{Y}$ have degree at least $4$.

\begin{lemma}
Let $v\in \partial \thick{Y}^0$ and $x\coloneqq\pi(v)\in\partial Y$, i.e.\ $vx$ is an edge.
Then,
\begin{enumerate}[noitemsep]\label{lem:thickdegrees}
\item If $\deg_{\partial\thick{Y}}(v)=3$, then $\deg_{Y}(x)=3$. 
Conversely, for every $x\in\partial Y^0$ with $\deg_{\partial Y}(x)=3$, $\pi^{-1}(x)$ is a singleton vertex and $\deg_{\partial\thick{Y}}(\pi^{-1}(v))=3$.
\item If $\deg_{\partial\thick{Y}}(v)\geq 4$, then $\deg_{\partial Y}(x)\geq 4$.
\item Suppose $\deg_{\partial \thick{Y}}(v)=4$ and the faces of $\partial \thick{Y}$ incident to it are, in cyclic order, $F_1, \dotsc, F_4$ with projections $\pi(F_1),\dotsc, \pi(F_4)\subseteq Y$. Then $\pi(F_1),\dotsc, \pi(F_4)$ are distinct and if $\pi(F_1)$, $\pi(F_2)$ lie in a common cube $C$ of $Y$, then $\pi(F_3)$, $\pi(F_4)$ also lie in a common cube $D$ of $Y$, and $\deg_{\partial Y}(x)=4$.
\end{enumerate}
\end{lemma}
\begin{proof}
Note that links of vertices in $Y$ do not contain self-loops nor double edges as they are simplicial complexes.
In particular, in light of Fact~\ref{fact:thicklink}, this guarantees that all vertices of $\partial\thick{Y}$ have degree at least $3$: vertices of degree $1$ or $2$ would imply self-loops and double edges in boundary links of $Y$, respectively.

Suppose $\deg_{\partial\thick{Y}}(v)=3$. 
Then, since $\lk_{\partial Y}(\pi(v))$ does not contain self-loops nor double edges, we have that $\pi^*(\lk_{\partial\thick{Y}}(v))$ is a triangle.
Since $Y$ is \cato, it follows that this triangle corresponds to a cube $C\subseteq Y$, with faces $F_1,F_2,F_3$ incident to $x$.
By construction a small ball around $x$ in $\partial \thick{Y}$ is covered by $C\cup C_{F_1}\cup C_{F_2}\cup C_{F_3}$, so $\deg_{\thick{Y}}(x)=4$ and $\deg_{Y}(x) = 3$.
The converse follows as $Y$ is clean: vertices of degree $3$ on $\partial Y$ are corners of cubes, so the claim follows by construction. This also proves the second bullet point as there are no vertices of degree less than $3$ in $\partial \thick{Y}$.

Suppose $v\in\partial\thick{Y}^0$ satisfies the assumptions of point 3.
Then since $\lk_{\partial Y}(\pi(v))$ does not contain self-loops nor double edges, we have that $\pi^*(\lk_{\partial\thick{Y}}(v))$ is a cycle of length $4$, so the projected faces $\pi(F_1), \dotsc, \pi(F_4)$ are distinct.
As $\pi(F_1)$ and $\pi(F_2)$ lie in the same cube of $Y$, by Lemma~\ref{lem:basiclinks} \ref{lem:basiclinks3}, $\pi(v)$ is incident to exactly two cubes in $X$, has $\deg_{\partial Y}(x)=4$ and $\pi(F_3)$, $\pi(F_4)$ lie in the same cube, as desired.
\end{proof}

\begin{lemma}\label{lem:thicksctooriginalsc} 
There is a one to one correspondence between row configurations in $\partial \thick{Y}^1$ and $\partial Y^1$.
Namely if $p_0p_1\dotsc p_k$ is the spine of a row configuration in $\partial \thick{Y}^1$ then $\pi(p_0)\pi(p_1)\dotsc \pi(p_k)$ is the spine of a row configuration in $\partial Y^1$.
\end{lemma}
\begin{proof}
By Lemma~\ref{lem:thickdegrees}.1 and since $Y$ is clean, $\deg_{\partial Y}(\pi(p_0)) = 3$, so $\pi(p_0)$ is the corner of a cube $C_1$.
Now, $\deg_{\partial\thick{Y}}(p_1)=4$ and $p_1$ has a cyclic ordering of incident faces $F_1, \dotsc, F_4$.
Since $\pi$ is combinatorial, $\pi(p_0)\pi(p_1)$ is an edge and so $\pi(p_1)\in C_1$.
Hence, without loss of generality we may assume $\pi(F_1),\pi(F_2)\subseteq C_1$ and so by Lemma~\ref{lem:thickdegrees}.3, $\pi(F_3),\pi(F_4)$ both lie in a cube $C_2$.
Moreover, with this labelling of faces, $p_2\in F_3,F_4$ so $\pi(p_2)\in C_2$.
Continuing in this way we find a row of cubes $(C_1,\dotsc,C_k)$ where $\pi(p_0)$ is a corner of $C_0$ and $\pi(p_{i-1}),\pi(p_i) \in C_i$ for each $1\leq i \leq k$.
This implies that $\pi(p_0),\dotsc,\pi(p_k)$ is the spine of a row configuration in $Y$, as desired.
\end{proof}
For simplicity, if $\mathcal{C}\subseteq(\partial\thick{Y})^1$ is a row configuration, we write $\pi(\mathcal{C})\subseteq (\partial Y)^1$ for its corresponding projected row configuration in $\partial Y$.

\subsection{Existence of good row configurations}
Let $Y$ be a clean \cato\ cube subcomplex of a cubulation of the ball $B^3$.
In this section we show that good row configurations exist in $(\partial Y)^1$.
By the correspondence between row configurations in $(\partial Y)^1$ and those in $G(Y)$ discussed in Section~\ref{sec:scs}, this implies that there are good configurations in $G(Y)$.
By Lemma~\ref{lem:thicksctooriginalsc}, our argument boils down to showing that good row configurations exist in the thickening $(\partial \thick{Y})^1$.
For this, we use a simple path-counting argument.

\begin{fact}\label{lem:sc1} 
Let $v\in \partial Y^0$ have $\deg_{\partial Y}(v)=3$.
Then, there are at least three row configurations in $(\partial Y)^1$ with $v$ as a starting point.
\end{fact}
The above fact is clear from definitions and can be strengthened to exactly three, but this will not be needed for our purposes.

\begin{figure}[ht]
	\centering
    \includegraphics[scale=0.8]{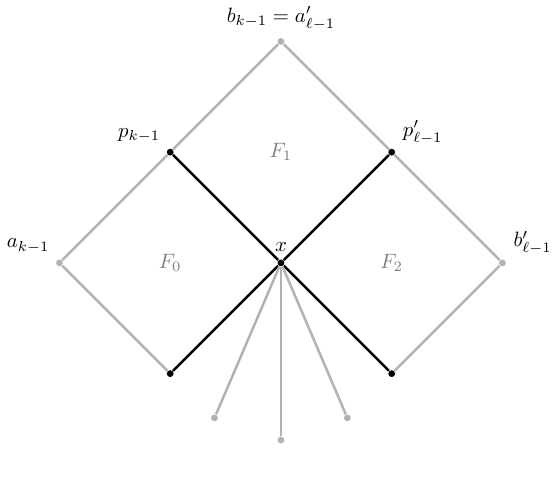}
	\caption{Two row configurations sharing endpoint $x$. In fact, $F_0$ and $F_2$ must share an edge.}
    \label{fig:twoSCcontradiction}
\end{figure}

\begin{lemma}\label{lem:scdegbound}
Let $x\in\partial\thick{Y}^0$ be a vertex with $\deg_{\partial\thick{Y}}(x)\geq 5$.
Then there are at most $2\lfloor \frac{d}{3} \rfloor$ row configurations which end at $x$.
\end{lemma}
Note that the condition on the degree is necessary: one can easily find examples of $Y$ with vertices of degree $3$ having three row configurations ending at them in the thickening, e.g. when $Y$ a single cube.
\begin{proof}
Let
\begin{align*}
\mathcal{C} &= (a_i, p_i, b_i \colon i=0,\dotsc,k)\\
\mathcal{C}' &= (a_i', p_i', b_i' \colon i=0,\dotsc,\ell)
\end{align*}
be two row configurations in $\partial\thick{Y}^1$ with a common endpoint $x\coloneqq p_{k}=p_\ell'$, as in Figure~\ref{fig:twoSCcontradiction}.

If $p_{k-1} = p_{\ell-1}'$, then $\mathcal{C}$ and $\mathcal{C}'$ share the last edge of their spine, and by a simple inductive argument it follows that the two spines are the same, meaning $\mathcal{C} = \mathcal{C}'$. 
Suppose now that $p_{k-1} \neq p_{\ell-1}'$, yet they both lie in face $F_1$ with vertices $\{p_{k-1},x,p_{\ell-1}',y\}$ for $y=b_{k-1} = a_{\ell-1}'$, labelled as in Figure~\ref{fig:twoSCcontradiction}.
In particular, $F_0$, $F_1$ and $F_2$ are pairwise distinct faces of $\partial Y$.

By Lemma~\ref{lem:thicksctooriginalsc}, the row configurations $\mathcal{C}$, $\mathcal{C}'$ in $\partial\thick{Y}^1$ correspond to row configurations $\pi(\mathcal{C}),\pi(\mathcal{C}')$ in $\partial Y^1$ respectively, which share the face $\pi(F_1)$.
Let $(C_1,\dotsc, C_k)$, $(C_1', \dotsc, C_\ell')$ be the rows of cubes corresponding to $\pi(\mathcal{C})$, $\pi(\mathcal{C}')$, respectively.
Then, $C_k=C_{\ell}'$ as they share a face on $\partial Y$, 
Let $\{\pi(p_{k-1}), \pi(p'_{\ell-1}),z\}$ be the three neighbours of $\pi(x)$ in $C_k$.
\begin{claim}
There is no row configuration in $(\partial Y)^1$ whose spine ends with edge $\pi(x)z$.
\end{claim}
\begin{poc}
Since $\pi(F_0), \pi(F_1),\pi(F_2)\subseteq \partial Y$ lie in a common cube $C_k$ of $Y$, the edges corresponding to them in $\lk_{\partial Y}(\pi(x))$ form a triangle $T$.
We first show that there is a face $F$ in $\partial Y$ distinct from $\pi(F_0),\pi(F_1),\pi(F_2)$ which is incident to $\pi(x)z$.

By Lemma~\ref{lem:thickdegrees}.2, $\deg_{\partial Y}(\pi(x))\geq 4$ as $\deg_{\partial \thick{Y}}(x)\geq 5$ and so there is an edge $\pi(x)v$ in $Y$ such that $v\not\in \{ \pi(p_{k-1}), \pi(p'_{\ell-1}),z\}$.
Since $Y$ is clean, in particular $\pi(x)$ is not a cut vertex.
Fix a path $\gamma\subseteq Y^1\setminus \pi(x)$ from $v$ to $z$.
Notice that the edges $\pi(x)\pi(p_{k-1})$ and $\pi(x)\pi(p'_{\ell-1})$ are only incident to faces of $C_k$ in $Y$ as they are on the spine of row configurations.
Hence, in a disc diagram $D\rightarrow Y$ with $\partial D=\gamma$, since $\pi(x)z\subseteq C_k$ and $\pi(x)v\not\subseteq C_k$ the face $F$ incident to $\pi(x)z$ in $D$ is not on $C_k$, as claimed.

This suffices to prove the claim, since if $\pi(x)z$ was the last edge of the spine of a row configuration, then it would only be incident in $Y$ to the two faces of the last cube in the corresponding row of cubes.
\end{poc}
Indeed, assuming the claim, no row configuration of $\partial\thick{Y}^1$ can have spine ending with edge $z'x$, for $z'\in \pi(z)$, as this would lead to a contradiction in light of Lemma~\ref{lem:thicksctooriginalsc}.

In what follows, we restrict our considerations to row configurations in $\partial \thick{Y}^1$ which end in $x$, and faces of $\partial\thick{Y}$ incident to $x$.
Let $k$ be the total number of row configurations and write $k=t+2r$ where $t$ counts \defn{single} row configurations -- those sharing no face (incident to $x$) with other row configurations, and $r$ counts the \defn{pairs} of row configurations as above.
Then, each single row configuration forbids two faces while each pair forbids three.
In conclusion, 
\[
k = t+2r \leq \floor{\frac{d-3r}{2}} + 2r \leq 2\floor{\frac{d}{3}},
\]
as $3r\leq d$, concluding the proof.
\end{proof}

\begin{lemma}\label{lem:goodscshell}
There exists a good configuration in $\thick{Y}$.
\end{lemma}
\begin{proof}
Let $n_k$ denote the number of vertices on $\partial \thick{Y}$ with degree $k$ in $\partial \thick{Y}$, and $|E|$, $|F|$ denote the number of edges, faces respectively of $\partial \thick{Y}$.
Then, by Euler's formula and Corollary~\ref{cor:thickenball},
\begin{align*}
2 &= \sum_{k\geq 3} n_k - |E| +|F|.\\
\intertext{Note that every edge is in exactly two faces, and every face is bounded by exactly four edges. Hence $2|F| = |E|$ and so,}
2 &= \sum_{k\geq 3} n_k -\frac{1}{2} |E|\\
n_3 &= 8 + \sum_{k\geq 5} (k-4) n_k,
\end{align*}
where the last line follows from using that $\sum k\cdot n_k = 2|E|$, multiplying both sides by $4$ and rearranging.

We may assume that there exists a vertex of degree at least $5$ on $\partial X$ as otherwise a good configuration exists immediately.
Let $u,v$ be vertices of $\partial \thick{Y}$ with $\deg_{\partial \thick{Y}}(v) = 3$ and $\deg_{\partial \thick{Y}}(u) = k\geq 5$. 
By Fact~\ref{lem:sc1}, $v$ is the starting point of three row configurations, and Lemma \ref{lem:scdegbound} implies $u$ is the ending point of at most $2\lfloor k/3\rfloor$ row configurations.
Hence, there are at least $3n_3$ row configurations, of which at most $\sum_{k\geq 5} 2\lfloor\frac{k}{3}\rfloor n_k$ end in a vertex of degree at least $5$.
Write $N$ for the number of row configurations with starting and ending point of degree $3$.
Using the above formula for $n_3$ we thus have
\begin{align*}
N &\geq 3n_3 -\sum_{k \geq 5} 2\left\lfloor\frac{k}{3}\right\rfloor n_k\\
&= 24 + \sum_{k \geq 5} \underbrace{\left(3k-2\left\lfloor\frac{k}{3}\right\rfloor  - 12\right)}_{\geq 0 \text{ for } k\geq 5} n_k\\
&\geq 24.\qedhere
\end{align*}
\end{proof}

\section{Finite \cato~cube complexes of dimensions 1 and  2}\label{sec:dim12}
In this section, we sketch the proof of the following strengthening of Theorem~\ref{thm:disk}.

\begin{reptheorem}{thm:dim12}
Let $X$ be a \cato\ cube complex of top dimension at most $2$ with finitely many cells and $D$ its matrix of pairwise distances between vertices on the combinatorial boundary of $X$.
Then, the combinatorial type of $X$ is reconstructible from $D$.
\end{reptheorem}

The arguments we use are very similar to those used in the proof of Theorem~\ref{thm:main}, with much less technical complications.
\begin{proof}[Proof sketch of Theorem~\ref{thm:dim12}]
If $k=1$, then $X$ is a tree and its boundary are its leaves, so reconstruction can be achieved by a simple inductive argument.
Indeed, the distance between the neighbour of a leaf $v$ and other leaves is one less than the distance between other leaves and $v$, and this neighbour is again a leaf after removing $v$ precisely if it is not on any geodesic between pairs of leaves.

If $k=2$, we may first assume that there are no vertices of degree at most $1$, as we may then repeatedly remove such vertices. 
In particular, without loss of generality, every vertex of $X$ is contained in at least one face.
Similarly we may assume that $X$ has no cut vertices.

Consider an inclusion maximal disc diagram $D \hookrightarrow X$. 
In particular, since $X$ has no $3$-cubes and is \cato, $D$ is a quadrangulation with all internal degrees at least $4$, so \cite[Lemma 3.4]{john2D} implies that there is a vertex $v$ of degree $2$ on $\partial D$.
Since $X$ has no cut vertices and $D$ is maximal, $v$ is contained in a unique face $F$ of $X$. 
Let $u$ be the vertex not adjacent to $v$ in this face.
Then, by using Lemma~\ref{lem:adjacentverticesandhyperplanes} similarly as in the argument for the $3$-dimensional setting, we can recover distances between boundary vertices and $u$ in the new complex.

Lastly, we show that vertices on $\partial X$ with degree $2$ in $X$ can be recognised. 
Let $v$ be a vertex on $\partial X$ with $d_{\partial X} (v) = 2$ and $e$, $f$ be the two edges on $\partial X$ incident to $v$.
Write $H_e$ and $H_f$ to denote the hyperplanes in $X$ dual to $e$, $f$ respectively.
We will show that $v$ has degree $2$ in $X$ if and only if $H_e$ and $H_f$ cross in $X$.

Suppose first that $H_e$ and $H_f$ cross in $X$.
Now, by \cite[Lemma 3.6]{wise} $e$ and $f$ are contained in a same face $F$ of $X$.
In particular, since $e$ and $f$ are both on the boundary, $v$ is only incident to $F$ and therefore has degree $2$ in $X$.
Conversely, if $d_{X}(v)= 2$, $H_e$ and $H_f$ clearly cross in the face incident to $v$. 

We now explain how this condition can be recognised from the boundary distances.
Both $H_e$ and $H_f$ split $X$ into two connected components.
Let $\partial X^0 =A_e\cup B_e$, $A_f\cup B_f$ be the resulting partitions of the vertices of $\partial X$.
These can be identified from the boundary distances by using Lemma~\ref{lem:adjacentverticesandhyperplanes}.
If $H_e$ and $H_f$ do not cross in $X$, then it must be that, without loss of generality, $A_e\subseteq A_f$ or $B_e\subseteq A_f$.
Hence, checking that $H_e$ and $H_f$ cross in $D$ amounts to checking that that $A_e \not\subseteq A_f$ and $B_e\not\subseteq A_f$.
\end{proof}

\newpage
\section{Further work}
\label{sec:conclusion}

We conjecture that the \cato\ property is also necessary and sufficient for higher dimensions.
\begin{conjecture}\label{conj:higherdim}
For any $k\geq 4$, any finite \cato\ cube complex with an embedding in $\R^k$ can be reconstructed up to combinatorial type from its boundary distances.
\end{conjecture}
It is possible that our proof method can be generalised to confirm this conjecture. In particular, the definition of a row configuration can be generalised to correspond to rows of $k$-cubes, arguments using hyperplanes still apply, and the topological tools used have higher dimensional alternatives. On the other hand, one key sticking point is in generalising the thickening procedure: while we used the fact that if a clean 3-dimensional cube complex is not pure then this must be due to free faces, in $k$ dimensions there are more potential low-dimensional substructures to consider. It might then be necessary to expand the collection of cleaning operations used. In addition, we take advantage of 2-dimensional links in Section~\ref{sec:links} as well as results such as Lemma~\ref{lem:thickdegrees} where we draw conclusions from specific low-degree configurations. Higher dimensional analogues to these results are likely to be quite a lot more complicated.

Our motivation in this paper was to obtain a direct generalisation of Theorem~\ref{thm:disk} in $3$-dimensional space: that \cato\ cubulations of $3$-dimensional balls are reconstructible from their boundary distances.
In our attempt we ended up needing to show this statement for a wider class, \cato\ cube complexes with Euclidean embeddings.
Revisiting the $1$ and $2$-dimensional case with this in mind and with the appropriate notion of boundary, we showed in Theorem~\ref{thm:dim12} that Theorem~\ref{thm:disk} can be extended to general finite \cato\ cube complexes of top dimension $1$ and $2$. 
This prompts the natural question: is an embedding in a Euclidean space of the same dimension as the top dimension of the complex required in general?
Of course, such an embedding is not always guaranteed: take for instance a single $2$-cube with three $3$-cubes glued onto it.
We believe that such complexes are still reconstructible.
\begin{conjecture}\label{conj:higherdimgeneral}
For any $k\geq 3$, any finite \cato\ cube complex can be reconstructed up to combinatorial type from its boundary distances.
\end{conjecture}
For $k=3$, a key sticking point of an approach similar to the proof of Theorem~\ref{thm:main} is in the `thickening' step: not only is the definition of a `thickening' of a complex compromised when an embedding is not provided (the gluing maps between new cubes depends on the cyclic ordering of faces around edges in the embedding), but as long as that a `thickening' contains the original complex as a subcomplex, it will not be homeomorphic to an euclidean ball if the original complex was not embeddable in euclidean space, a crucial requirement in the existence argument.

Another direction is to consider the situation in the world of simplicial complexes.
\begin{question}
Under what conditions is a finite $k$-dimensional simplicial complex reconstructible up to combinatorial type from its boundary distances?
\end{question}
Reconstructibility for simplicial complexes is less approachable as at first sight there does not seem to be any obvious `convex' substructure to remove. The case of $2$-dimensional complexes embeddable in $\mathbb{R}^2$ has been dealt with in \cite{john2D}, where it is shown that all internal vertices having degree at least $6$ is a sufficient condition.

\bibliographystyle{scott}
\bibliography{cubulations}

\end{document}